\documentclass[psamsfont,a4paper,12pt]{amsart}
\usepackage[dvipsnames]{xcolor}
\usepackage{fullpage, amsmath, amssymb, amsthm, amscd, setspace, bm, graphicx, indentfirst, multirow, tikz, enumerate, verbatim, appendix}
\usepackage{adjustbox,amsfonts,array,graphicx,booktabs,tabularx,multirow,multicol,stmaryrd, tabu}
\usepackage[utf8]{inputenc}
\usepackage{cite}
\usepackage{hyperref}
\usepackage{mathabx,epsfig}
\usepackage{mathtools}
\usepackage{shuffle}
\usepackage{tikz-cd}
\usepackage{thmtools, thm-restate}

\date{\today}

\newtheorem{theorem}{Theorem}[section]
\newtheorem{lemma}[theorem]{Lemma}
\newtheorem{proposition}[theorem]{Proposition}
\newtheorem{corollary}[theorem]{Corollary}
\newtheorem{conjecture}[theorem]{Conjecture}
\newtheorem{problem}[theorem]{Problem}

\theoremstyle{definition}
\newtheorem{definition}[theorem]{Definition}
\newtheorem{remark}[theorem]{Remark}

\theoremstyle{remark}
\newtheorem{example}[theorem]{Example}

\newcommand{\Z}{\mathbb{Z}}
\newcommand{\Zcol}{\overline{\mathbb{Z}}}

\newcommand{\Q}{\mathbb{Q}}
\newcommand{\QAlg}{\mathcal{Q}}

\newcommand{\M}{\mathcal{M}}
\newcommand{\N}{\mathbb{N}}

\newcommand{\F}{\mathfrak{F}}
\newcommand{\olF}{\ol{\F}\mspace{-4mu}}
\newcommand{\olK}{\ol{\kappa}}
\newcommand{\mcR}{{\mathcal{R}}}
\newcommand{\mcS}{{\mathcal{S}}}

\newcommand{\bs}{\boldsymbol}
\newcommand{\bsvec}[1]{\bs{\hat{#1}}}
\newcommand{\mc}{\mathcal}

\newcommand{\ol}{\overleftarrow}
\newcommand{\Poly}{\text{Poly}}
\newcommand{\Sch}{\mathfrak{S}}

\newcommand{\RW}{\mathcal{RW}}
\newcommand{\Schvec}{\bsvec{\textbf{Sch}}}
\newcommand{\ZWInc}{\Z_{\text{WInc}}}
\newcommand{\IncSuf}{\text{IncSuf}}
\newcommand{\QSym}{\text{QSym}}
\newcommand{\BR}{\mathcal{BR}}
\newcommand{\code}{\text{code}}
\newcommand{\dualcode}{\text{dualcode}}
\DeclareMathOperator{\FS}{FS}
\DeclareMathOperator{\BS}{BS}
\DeclareMathOperator{\St}{St}
\DeclareMathOperator{\supp}{supp}
\renewcommand{\supp}{\mathop{\mathrm{supp}}\hspace{0.15em}}

\newcommand{\val}{\text{val}}
\newcommand{\maxcol}[1]{\text{maxcol}(#1)}
\newcommand{\col}[1]{\text{color}(#1)}
\newcommand{\til}[1]{\widetilde{#1}}

\newcommand\ceq{\stackrel{\mathclap{\normalfont\mbox{c}}}{=}}
\newcommand{\nilsim}{{\mathop{\overset{\raisebox{-0.3ex}{$\bullet$}}{\sim}}}}

\newcounter{Comment}

\title{When do Schubert polynomial products stabilize?}

\author{Andrew Hardt}
\address[A.~Hardt]{Department of Mathematics, University of Illinois Urbana-Champaign, Urbana, IL 61801}
\email{ahardt@illinois.edu}
\urladdr{https://andyhardt.github.io/}

\author{David Wallach}
\address[D.~Wallach]{Department of Mathematics, University of Illinois Urbana-Champaign, Urbana, IL 61801}
\email{davidrw3@illinois.edu}

\begin{document}

\begin{abstract}
The \emph{back-stabilization number} for products of Schubert polynomials is the distance the corresponding permutations must be ``shifted'' before the structure constants stabilize. We give an explicit formula for this number and thereby prove a conjecture of N.~Li in a strengthened form. This leads to an additional result: a formula for the smallest $n$ such that a given Schubert product expands completely over $S_n$.

Our method is to explore back-stable fundamental slide polynomials and their products combinatorially, in the context of their associated \emph{words}. We use three main tools: (i) an algebra consisting of \emph{colored words}, with a modified \emph{shuffle product}, and which contains the rings of back (quasi)symmetric functions as subquotients; (ii) the combinatorics of \emph{increasing suffixes} of reduced words; and (iii) the lift of differential operators to the space of colored words.
\end{abstract}

\maketitle

\section{Introduction}

\emph{Schubert polynomials} are polynomial representatives $\Sch_w(x_1,\ldots,x_n)$ for the classes of Schubert varieties in the cohomology ring of the flag variety $GL_n/B$. They were first defined by Lascoux and Sch\"utzenberger \cite{LascouxSchutzenberger}, and satisfy nice combinatorial and geometric properties. They are indexed by permutations in $S_n$, the symmetric group on $n$ letters, and satisfy the stability property $\Sch_w = \Sch_{w\times 1}\in S_{n+1}$, where $w\times 1$ is the permutation acting by $w$ on $1,\ldots, n$ and fixing $n+1$. Thus, we can take $w$ to be an element of $S_{\Z_+}$, the set of bijections $\Z_+\to\Z_+$ fixing all but finitely-many elements. Schubert polynomials form a basis of the polynomial ring $\Q[x_1,x_2,\ldots]$, and their \emph{Schubert structure constants} $c_{u,v}^w$, given by
\begin{equation*}
\Sch_u\Sch_v = \sum_w c_{u,v}^w \Sch_w,
\end{equation*}
are known to be nonnegative integers for geometric reasons \cite{Fulton-Young-tableaux}. However, it remains a longstanding open problem to give a combinatorial proof of this fact.

The first proven monomial expansion of Schubert polynomials was given by Billey, Jockusch, and Stanley \cite{BilleyJockuschStanley}. Given a reduced word $\bs{w}$, let $\mc{C}(\bs{w})$ be the set of \emph{compatible sequences} with top row $\bs{w}$ (Definition~\ref{def:compatible-sequences}), and let $\mc{C}^+(\bs{w})$ be the subset with positive entries. Billey--Jockusch--Stanley showed:
\begin{equation} \label{eq:bjs-formula}
\Sch_w = \sum_{\substack{\bs{w} \text{ reduced}\\ \text{word for $w$}}} \F_{\bs{w}}, \qquad\qquad \text{where} \qquad\qquad \F_{\bs{w}} = \sum_{\bs{\alpha}\in\mc{C}^+(\bs{w})}x^{\bs{\alpha}}.
\end{equation}

Assaf and Searles studied $\mathfrak{F}_{\bs{w}}$ in \cite{AssafSearles}, where it is called the \emph{(fundamental) slide polynomial}.\footnote{In \cite{AssafSearles}, the $\F_{\bs{w}}$ were indexed by integer compositions; see \cite[Example~4.13]{NadeauTewari} for a direct comparison.} Slide polynomials form a basis of the polynomial ring, and while Schubert polynomials are very hard to multiply, the multiplication of slide polynomials is manifestly positive and has a simple formula involving the shuffle product of Eilenberg and MacLane \cite{EilenbergMacLane}.

Now let $w$ be an element of $S_\Z$, the set of bijections $\Z\to\Z$ that fix all but finitely many elements. Let $\gamma(w)$ be the unit shift of $w$, given by $\gamma(w)(i) = w(i-1)+1$. Define
\begin{equation*}
V_{k}(u,v) = \left\{w \in S_{\Z} \;|\; c_{\gamma^k(u),\gamma^k(v)}^{\gamma^k(w)} \ne 0\right\}.
\end{equation*}

Li \cite{Li-back-stable-conjecture} showed that when $u,v,w \in S_{\Z_+}$, $c_{\gamma(u),\gamma(v)}^{\gamma(w)} = c_{u,v}^w$, so $V_k(u,v) \subseteq V_{k+1}(u,v)$.

\begin{definition}
\label{def:back-stabilization}
\;
    \begin{itemize}
        \item The \emph{back-stabilization number} $\BS(u,v)$ of the pair $(u,v)$ is the smallest nonnegative integer $j$ such that $V_{j}(u,v) = V_{j+1}(u,v) = V_{j+2}(u,v) = \ldots$.
        \item The \emph{stability number} $\St(u,v)$ of $(u,v)$ is the smallest nonnegative integer $j$ such that $V_j(u,v) = V_{j+1}(u,v)$.
    \end{itemize}
\end{definition}

Li showed that $\BS(u,v)$ is always finite. Clearly, $\BS(u,v)\ge \St(u,v)$, and Li conjectured that in fact they are equal. For $w\in S_\Z$, let
\[\theta_i(w) = \begin{cases}
    1 & \text{if }\exists j>i \text{ with } w(i)>w(j), \\
    0 & \text{otherwise}.
\end{cases}
\]

Note that $\theta_i(w)$ is $1$ if and only if the $i$th entry of the \textit{Lehmer code} (Definition \ref{def:lehmer-code}) of $w$ is non-zero.

\begin{conjecture}[{\!\!\cite[Conjecture~1.5]{Li-back-stable-conjecture}}]
 \label{conj:li-back-stable} \;
    \begin{enumerate}
        \item[(a)] $\BS(u,v) = \St(u,v)$ for all $u,v$.
        \item[(b)] $\BS(u,v) \le \max(\{i \;|\; \theta_i(u) = 1 \text{ or } \theta_i(v) = 1\})$.
    \end{enumerate}
\end{conjecture}

Our first main result is a strengthened form of Conjecture \ref{conj:li-back-stable}, including a precise formula for $\BS(u,v)$. Let 
\[
\lambda_i(w) = \sum_{j\le i} \theta_j(w),
\]
the number of nonzero rows of the Lehmer code of $w$ with index at most $i$.

\begin{theorem}[Back-Stabilization Theorem]
 \label{thm:backstabilization-thm} \;
 Conjecture~\ref{conj:li-back-stable} is true. Moreover, for all $u,v\in S_{\Z_+}$,
 \begin{equation} \label{eq:bs-formula}
     \BS(u,v) = \max_{i \geq 0}(\lambda_i(u) + \lambda_i(v) - i).
 \end{equation}
\end{theorem}

\eqref{eq:bs-formula} implies Conjecture \ref{conj:li-back-stable}(b) because for all $i$, $\lambda_i(v)\le i$, so
\[
\lambda_i(u) + \lambda_i(v) - i \le \lambda_i(u) \le \lambda_\infty(u) = \lambda_{\max(\{i \;|\; \theta_i(u) = 1\})}(u) \leq \max(\{i \;|\; \theta_i(u) = 1\}).
\]

Li proved \cite[Theorem~1.2]{Li-back-stable-conjecture} that Conjecture \ref{conj:li-back-stable} is true whenever either $u$ or $v$ is a Grassmannian permutation. In addition, for arbitrary $u$ and $v$ she proved the bound $\BS(u,v) \le \ell(u) + \ell(v)$ \cite[Theorem~1.3]{Li-back-stable-conjecture}, where $\ell(u)$ refers to the \textit{Coxeter length} of $u$ \cite{BjornerBrenti}. This bound was later improved by Assaf and Searles \cite[Corollary~5.17]{AssafSearles}.

Work of Lam, Lee, and Shimozono \cite{LamLeeShimiozono-schubert} shows that it is productive to work in the \emph{back-stable} setting. For any $w\in S_\Z$, define
\[
\ol{\Sch}_w = \sum_{\substack{\bs{w} \text{ reduced}\\ \text{word for $w$}}} \olF_{\bs{w}}, \qquad\qquad \text{where} \qquad\qquad \olF_{\bs{w}} = \sum_{\bs{\alpha}\in\mc{C}(\bs{w})} x^{\bs{\alpha}}.
\]
$\ol{\Sch}_w$ is the \emph{back-stable Schubert polynomial} \cite{LamLeeShimiozono-schubert}. The back-stable Schubert polynomials form a basis of the ring of \emph{back-symmetric functions}. $\olF_{\bs{w}}$ is the \emph{back-stable slide polynomial}. The back-stable slide polynomials were first defined by Nadeau and Tewari \cite{NadeauTewari} and they form a basis of \emph{back-quasisymmetric functions}.\footnote{Note that none of the back-stable functions we consider are strictly polynomials; however, they share many properties in common with polynomials, and we follow \cite{LamLeeShimiozono-schubert} in using this terminology.}

Let $\ol{c_{u,v}^w}$ be the \emph{back-stable Schubert structure constant} defined by
\begin{equation*}
\ol{\Sch}_u\ol{\Sch}_v = \sum_w \ol{c_{u,v}^w} \ol{\Sch}_w.
\end{equation*}
The ordinary and back-stable structure constants
determine each other. For any $w\in S_\Z$, if $k$ is large enough, $\gamma^k(w)\in S_{\Z_+}$; let $\BS(w)$ denote the smallest such nonnegative~$k$. Then if $u,v\in S_{\Z_+}, w\in S_\Z$,
\[
c_{u,v}^w = \begin{cases} \ol{c_{u,v}^w}, & \text{if } u,v,w\in S_{\Z_+}, \\ 0, & \text{otherwise}, \end{cases}
\qquad \text{and} \qquad
\ol{c_{u,v}^w} = c_{\gamma^k(u),\gamma^k(v)}^{\gamma^k(w)} \quad \text{for any } k\ge \BS(w).
\]

An immediate consequence of Theorem \ref{thm:backstabilization-thm} is a characterization of when $\Sch_u \Sch_v$ has a Schubert expansion identical to the back-stable Schubert expansion of $\ol{\Sch}_u\ol{\Sch}_v$.
\begin{corollary} For all $u,v\in S_{\Z_+}$,
\[
    \ol{c_{u,v}^w} = c_{u,v}^w \text{ for all } w
    \iff c_{u,v}^w = c_{\gamma(u),\gamma(v)}^{\gamma(w)} \text{ for all } w
    \iff \lambda_i(u) + \lambda_i(v) \leqslant i \text{ for all } i\ge 0.
\]
\end{corollary}

\begin{example}
Let $u = 321, v = 213\in S_3\subseteq S_{\Z_+}$, expressed in one-line notation. We have $\lambda_1(u) = 1$, $\lambda_i(u)=2$ for $i\ge 2$, and $\lambda_i(v) = 1$ for $i\ge 1$; in particular, $\max(\{i \;|\; \theta_i(u) = 1\}) = 2$ and $\max(\{i \;|\; \theta_i(v) = 1\}) = 1$, so Conjecture~\ref{conj:li-back-stable}(b) says that $\BS(u,v)\le 2$. On the other hand, by \eqref{eq:bs-formula}, $\BS(u,v) = \max_{i \geq 0}(\lambda_i(u) + \lambda_i(v) - i) = 1$ (and we only need to check $i\le 2$ since $u$ and $v$ fix every integer $\ge 3$). Indeed, computing the back-stable Schubert product,
\[
\ol{\Sch}_{321}\ol{\Sch}_{213} = \ol{\Sch}_{4213} + \ol{\Sch}_{13204} + \ol{\Sch}_{23014}.
\]
On the right side, $4213$ is in $S_{\Z_+}$, while the other permutations are not, so we have $\Sch_{321}\Sch_{213} = \Sch_{4213}$. Applying $\gamma$, $\Sch_{1432}\Sch_{1324} = \Sch_{15324} + \Sch_{24315} + \Sch_{34125}$, and the product has stabilized. Applying $\gamma$ again yields no new permutations: $\Sch_{12543}\Sch_{12435} = \Sch_{126435} + \Sch_{135426} + \Sch_{145236}$.
\end{example}

Theorem~\ref{thm:backstabilization-thm} leads us to consider a different Schubert calculus problem, which turns out to be closely related. For $w\in S_{\Z_+}$, let the \emph{forward-stability number} for $w$ be the integer
\[
\FS(w) = \min\{n\ge 1 \;|\; w\in S_n\},
\]
and if $u,v\in S_{\Z_+}$, let
\begin{equation} \label{eq:forward-stability-def}
\FS(u,v) = \max_{w|c_{u,v}^w\ne 0} \FS(w).
\end{equation}
$\FS(u,v)$ is the smallest integer $n$ such that the Schubert product $\Sch_u\Sch_v$ can be fully realized in the flag variety $GL_n/B$.

Let $\Lambda_i(w) = |\{j\ge i \;|\; \exists j'<j \text{ with } w(j')>w(j)\}|$, the number of nonzero entries of the \emph{dual Lehmer code} (Definition~\ref{def:dual-lehmer-code}) of $w$ with index at least $i$.

\begin{theorem}[Forward-Stability Theorem] \label{thm:minimal-flag-variety}
    For all $u,v\in S_{\Z_+}$,
    \begin{equation} \label{eq:mfv-formula}
    \FS(u,v) = \max_{i\le 1+\max(\FS(u),\FS(v))}(\Lambda_i(u)+\Lambda_i(v)+i-1).
    \end{equation}
\end{theorem}

Theorem~\ref{thm:minimal-flag-variety} is in some sense a ``dual'' result to Theorem~\ref{thm:backstabilization-thm}. Schubert structure constants are symmetric under conjugation by $w_0$ (Lemma~\ref{lem:w0-cuvw-relation}), and as a consequence $\FS(u,v)$ is closely related to $\BS(w_0uw_0,w_0vw_0)$. However, Theorem~\ref{thm:backstabilization-thm} and Theorem~\ref{thm:minimal-flag-variety} are not equivalent, since the former involves back-stable structure constants while the latter involves ordinary ones.

As a direct consequence of Theorem~\ref{thm:minimal-flag-variety}, we have the following simple check for whether a Schubert product fully manifests in a given flag variety:

\begin{corollary}
    Let $u,v\in S_n$. Then,
    \[
    w\in S_n \text{ for all } w \text{ with } c_{u,v}^w\ne 0 \quad\iff\quad \Lambda_i(u) + \Lambda_i(v) \le n+1-i \text{ for all } 1\le i\le n.
    \]
\end{corollary}

\begin{example}
Let $u = 436521, v = 54312$. We have $\FS(u)=6, \FS(v)=5$, and
\[
(\Lambda_i(u))_{i=1,\ldots,7} = (4,4,3,3,2,1,0), \qquad (\Lambda_i(v))_{i=1,\ldots,7} = (4,4,3,2,1,0,0),
\]
\[
(\Lambda_i(u)+\Lambda_i(v)+i-1)_{i=1,\ldots,7} = (8,9,8,8,7,6,6),
\]
and so by Theorem~\ref{thm:minimal-flag-variety}, $\FS(u,v) = 9$, and indeed
\[
\Sch_{436521} \Sch_{54312} = \Sch_{86732145} + \Sch_{87532146} + \Sch_{965321478}.
\]
\end{example}

It is natural to seek a generalization of Theorem~\ref{thm:minimal-flag-variety} to products of Schubert polynomials in other classical Lie types \cite{BilleyHaiman}; further discussion may appear elsewhere.

Our follow-up paper \cite{HardtWallach-Grothendieck} shows that products of double Schubert, Grothendieck, and double Grothendieck polynomials have the same back- and forward-stabilization numbers as the corresponding products of Schubert polynomials.

\medskip

The remaining sections are organized as follows. Section \ref{sec:background} is background on words, permutations, and the shuffle product. In Section \ref{sec:word-polys}, we define colored words and the colored shuffle product, and give background on back-stable (quasi)symmetric functions. Section \ref{sec:colored-shuffle-algebra} defines our main object of study, the colored shuffle algebra $\QAlg$. We realize the rings of back-stable (quasi)symmetric functions as subquotients of this ring, discuss weight-preserving maps inside $\QAlg$, and finally, define lifts to $\QAlg$ of two differential operators, $\xi$ and $\nabla$, found in \cite{Stanley-schubert-shenanigans, Nenashev}.

Section \ref{sec:backstabilization-thm-proof} contains the proof of Theorem~\ref{thm:backstabilization-thm}. Much of the section is devoted to combinatorial properties of increasing suffixes, including one result (Theorem \ref{thm:increasing-suffix-product}) on the lengths of increasing suffixes which appear in a Schubert product, and another (Corollary \ref{cor:lehmer-nonzero-relationship}) bounding the number of nonzero rows in the Lehmer code. Finally, we combine these tools to prove Theorem~\ref{thm:backstabilization-thm}.

Section \ref{sec:forward-stabiization} contains the proof of Theorem~\ref{thm:minimal-flag-variety}. On the way, we prove a similar result (Corollary~\ref{cor:forward-stabilization}), which is in a precise sense dual to Theorem~\ref{thm:backstabilization-thm}. Corollary~\ref{cor:forward-stabilization} and Theorem~\ref{thm:minimal-flag-variety} only differ in that the former considers back-stable structure constants, and the latter ordinary structure constants, although to go between the two results involves most of the technical tools we develop in this paper.

The final two sections are on related topics. In Section \ref{sec:down-up-connectedness}, we define two notions of what we call ``down-up moves'', and conjecture that the sets of permutations which appear in a (back-stable) Schubert product are connected under these moves. These conjectures are partially-ordered by strength, we show that the stronger ones imply an alternate proof of part (a) of Conjecture~\ref{conj:li-back-stable}. In Section \ref{sec:back-stable-keys}, we define back-stable versions of \emph{key polynomials} as sums of slide polynomials, and prove that they are a basis for the space of back-symmetric functions.

\subsection*{Acknowledgements}

The authors are very grateful to Alexander Yong for numerous suggestions, including participation in the genesis of discussions which eventually led to this paper. Additionally, the authors thank Vic Reiner for helpful discussions. This work benefited from computations using \textsc{SageMath}.

Both authors were partially supported by NSF RTG grant DMS-1937241. D.~W. was supported through the ICLUE program at the University of Illinois Urbana-Champaign.

\section{Words and permutations} \label{sec:background}

\subsection{Words} \label{sec:words}
For some set $\mathcal{A}$, a \emph{word} over the alphabet $\mathcal{A}$ is a finite sequence $\bs{p} = (p_1, p_2, \ldots, p_k)$, where $p_i \in \mathcal{A}$ for all $1\le i\le k$; we will leave out the punctuation when convenient. We call $p_i$ the $i$th \emph{entry} of $\bs{p}$, and we will also denote this entry by $\bs{p}_i$. The \emph{support} of $\bs{p}$ is the set, $\supp{\bs{p}} = \{p_1,p_2,\ldots,p_k\}$, consisting of the entries of $\bs{p}$. The \emph{length} of $\bs{p}$, denoted $\ell(\bs{p})$, is simply the number of entries. The unique word of length $0$ is called the \textit{empty word} and is denoted with $\emptyset$. Let $\mathcal{A}^k$ be the set of words of length $k$, and let \[\mathcal{A}^* = \bigcup_{k\ge 0} \mathcal{A}^k.\]

If $\bs{q},\bs{r}\in\mathcal{A}^*$, then the \emph{concatenation} of $\bs{q}$ and $\bs{r}$ is the word $\bs{q} \circ \bs{r} = (q_1,\ldots,q_{\ell(b)},r_1,\ldots,r_{\ell(c)})$. $\bs{r} = (r_1,r_2,\ldots,r_m)$ is a \emph{subword} of $\bs{p}$ if $\bs{r} = (p_{i_1},\ldots, p_{i_m})$ with $i_1<i_2<\cdots<i_m$. If $\bs{p} = \bs{q}\circ \bs{r}$, then $\bs{q}$ is a \emph{prefix} of $\bs{p}$ and $\bs{r}$ is a \emph{suffix} of $\bs{p}$.

Defined by Eilenberg and Mac Lane \cite{EilenbergMacLane}, the \emph{shuffle product} is a linear operator on $\Q[\mathcal{A}^*]$, the vector space of formal finite $\Q$-linear combinations of elements of $\mathcal{A}^*$. For $\bs{p} = (p_1, \ldots, p_j) \in \mathcal{A}^*$ and $\bs{q} = (q_1, \ldots,q_j) \in \mathcal{A}^*$, the shuffle product of $\bs{p}$ and $\bs{q}$ is defined recursively as \[\bs{p} \shuffle \emptyset = \emptyset \shuffle \bs{p} = \bs{p},\] and \[(p_1, \ldots, p_j) \shuffle (q_1, \ldots,q_k) = p_1 \circ ((p_2 \ldots, p_j) \shuffle (q_1, \ldots, q_k)) + q_1 \circ ((p_1 \ldots, p_j) \shuffle (q_2, \ldots, q_k)).\]

$\bs{p}\shuffle\bs{q}$ is the formal sum of all words $\bs{r}\in\mathcal{A}^*$ which are obtained by interlacing the entries of $\bs{p}$ and $\bs{q}$ in an order-preserving way. We call such words \emph{shuffles} of $\bs{p}$ and $\bs{q}$. Note that the concatenation $\bs{p}\circ\bs{q}$ is one such shuffle, and every shuffle contains $\bs{p}$ and $\bs{q}$ as subwords. In any shuffle $\bs{r}$ of $\bs{p}$ and $\bs{q}$, call the entries of $\bs{r}$ taken from $\bs{p}$ the \emph{$\bs{p}$-entries}, and call the entries taken from $\bs{q}$ the \emph{$\bs{q}$-entries}.

\subsection{Permutations}

A particularly important context is when a word is a reduced word for a permutation. The basics of permutations and their Coxeter presentation are well-known. We give the relevant background below. For more details, see \cite{BjornerBrenti}.

A \emph{permutation} of a set $A\subseteq\Z$ is a bijection $w:A\to A$ such that $w(i)=i$ for all but finitely many $i\in A$. We write $S_A$ for the set of all permutations of $A$, and let $S_n := S_{\{1,\ldots,n\}}$. For all $i\in\Z$, let $s_i$ be the permutation with $s(i)=i+1, s(i+1)=i$, and for all other $j\in\Z$, $s(j)=j$. The $s_i$'s are called \emph{simple reflections}. They generate $S_\Z$ and satisfy the relations
\begin{subequations}
\label{eq:coxeter-relations}
\begin{equation}
\label{eq:quadratic-relation}
s_i^2 = id, \qquad \text{for all $i\in\Z$},
\end{equation}
\begin{equation}
\label{eq:braid-rel1}
s_is_{i+1}s_i = s_{i+1}s_is_{i+1} \qquad \text{for all $i\in\Z$},
\end{equation}
\begin{equation}
\label{eq:braid-rel2}
s_is_j = s_js_i \qquad \text{for all $i,j\in\Z, |i-j|\ge 2$}.
\end{equation}
\end{subequations}
This is called the \emph{Coxeter presentation} of $S_\Z$. \eqref{eq:quadratic-relation} are called the \emph{quadratic relations}, while \eqref{eq:braid-rel1},\eqref{eq:braid-rel2} are called the \emph{braid relations}.

If $w = s_{w_1}\cdots s_{w_k}$, we say that $\bs{w}:= (w_1,\ldots,w_k) \in \Z^*$ is a \emph{word} for $w$. If $\bs{w}$ has minimal length among all words for $w$, we call $\bs{w}$ a \emph{reduced word} for $w$, and the \emph{(Coxeter) length} of $w$ is $\ell(w) = \ell(\bs{w})$. Let $\RW(w)$ be the set of reduced words for $w$. It is a well-known classical result (Matsumoto's Theorem) that any two reduced words for $w$ are equivalent up to the braid relations, and any word for $w$ that is not reduced can be shortened to a reduced word via applications of the braid and quadratic relations. Let $\RW = \bigcup_{w \in S_{\Z}}$ be the set of reduced words for all permutations.

We write $u\le v$ if there exists a reduced word $\bs{v}$ for $v$ with a subword $\bs{u}$ that is a reduced word for $u$. (It turns out that if $u\le v$, \emph{every} reduced word for $v$ has such a subword.) $\le$ forms a partial order on $S_\Z$ called the \emph{(strong) Bruhat order}.

\begin{remark}
Most sources for permutations concern the finite groups $S_n$. Working in $S_\Z$ is a notational convenience, but for us has a deeper significance, in that back-stabilization is most natural in this context. Still, any particular computation involving permutations may be done in some $S_n$, by shifting if necessary, and so the theory of $S_\Z$ is not very different to the theory of $S_n$ for finite $n$.
\end{remark}

The \emph{one-line notation} of $w\in S_\Z$ is the infinite sequence $\ldots w(-2), w(-1), w(0), w(1), w(2), \ldots$. If $w\in S_n$, then we truncate the sequence on both sides and write $w(1), w(2), \ldots, w(n)$. In examples, $n$ is often a small integer, so we omit the commas as is standard practice. It is possible to interpret such a permutation as a word in $[n]^*$; however, for the most part we do \emph{not} want to do this.

\subsection{Lehmer code}
\begin{definition} \label{def:lehmer-code}
The \emph{Lehmer code} is the doubly-infinite sequence
\[
\text{code}(w) = (\ldots, c_{-1}, c_0, c_1, \ldots),
\]
where
\[
c_i := \text{code}(w)_i = |\{j\in \Z \;|\; j>i, w(j)<w(i)\}|.
\]
All but finitely many entries of the Lehmer code are zero, and $\ell(w) = \sum_i c_i$. $c_i$ is also called the \emph{$i$th row} of the Lehmer code, since it equals the number of boxes in the $i$th row of the Rothe diagram of $w$.
\end{definition}

We have
\[
\theta_i(w) = \begin{cases} 1, & \text{if code}(w)_i > 0,\\ 0, & \text{otherwise}.\end{cases} \qquad \text{and} \qquad \lambda_i(w) = |\{j\le i \;|\; \text{code}(w)_j > 0\}|.
\]

\begin{lemma} \label{lem:desdefs}
    The following are equivalent:
    
    \begin{enumerate}
        \item[(a)] $\ell(w s_i) = \ell(w) - 1$
        \item[(b)] $c_i > c_{i+1}$
        \item[(c)] $\text{code}(w s_i) = (\ldots, c_{i - 2}, c_{i - 1}, c_{i+1}, c_{i} - 1, c_{i+2}, c_{i+3}, \ldots)$
        \item[(d)] There exists a reduced word $(w_1, \ldots, w_k) \in \RW(w)$ with $w_k = i$
        \item[(e)] $w(i) > w(i+1)$
    \end{enumerate}
\end{lemma}

If any of the above are true, then $w$ is said to have a \textit{descent} at position $i$. This is sometimes called a \emph{right descent} since multiplication by $s_i$ happens on the right.

Recall the back-stabilization number $\BS(w)$ for $w$ is the smallest nonnegative integer $k$ such that $\gamma^k(w)\in S_{\Z_+}$ 

\begin{proposition}
Suppose $w\notin S_{\Z_+}$, and let $\bs{w}\in\RW(w)$. Then,
\begin{equation} \label{eq:BS-def}
\BS(w) = 1- \min(\supp \bs{w}) = 1 - \min(\{i \;|\; s_i \leq w\}) =  1- \min(\{j \;|\; \theta_j(w) = 1\}).
\end{equation}
\end{proposition}

\begin{proof}
Let $k = \BS(w)$. Then $w$ fixes $\cdots, -k-1, -k$, but not $1-k$, so $\supp\bs{w}$ cannot contain $\cdots, -k-1, -k$, but must contain $1-k$, and the first equality holds. We also have the second equality since $s_i\leq w$ if and only if it appears in a (equivalently, every) reduced word for $w$.

Finally, $w$ fixes $\cdots, -k-1, -k$, but not $1-k$, so $\theta_{1-k}(w)=1$ since if we let $i=1-k$ and $j = w^{-1}(1-k)$, then $w(i) > i = w(j)$ and $j > i$. Then for any $i \leq -k$, we have $\theta_i(w)=0$ because for any $j > i$, $w(i)=i<\min(j,1-k)\le w(j)$. So the third equality holds.
\end{proof}

\section{Colored words and slide polynomials} \label{sec:word-polys}

\subsection{Colored words and the colored shuffle product}
Let $\Zcol$ be the alphabet with letters $i^{[j]}$ with $i \in \Z$, $j \in \Z_+$. We refer to $\Zcol$ as the set of \textit{colored integers}, and $\Zcol^*$ as the set of \emph{colored words}.

Let the \textit{value} of $i^{[j]}$ be $\val(i^{[j]}) = i$, and let the \textit{color} of $i^{[j]}$ be $\col{i^{[j]}} = j$. $\Zcol$ has a total ordering, where we first compare values, then colors:
\[
\cdots < i^{[1]} < i^{[2]} < i^{[3]} < \cdots < (i+1)^{[1]} < (i+1)^{[2]} < (i+1)^{[3]} < \cdots
\]

For $\bs{p} = (p_1, \ldots, p_k) \in \Zcol^k$, let $\val(\bs{p}) = (\val(p_1), \ldots, \val(p_k)) \in \Z^k$ and let $\col{\bs{p}} = (\col{p_1}, \ldots, \col{p_k}) \in \Z_+^k$. Let $\maxcol{\bs{p}}$ be $\max(\col{\bs{p}})$. We let the \emph{color shift operator} $\uparrow$ be defined so that for $a \in \Z$,  $\bs{p} \uparrow a$ is the word such that $\val(\bs{p} \uparrow a) = \val(\bs{p})$ and $\col{\bs{p} \uparrow a} = (\col{p_1} + a, \ldots, \col{p_k} + a)$. We identify $\Z$ with the subset of $\Zcol$ of elements with color $1$, meaning $i := i^{[1]}$ for all $i$.

We have the following partial ordering on $\Zcol^k$:
\begin{equation} \label{eq:partial-ordering}
\bs{p} \le \bs{q} \iff p_j\le q_j \text{ for all } j,
\end{equation}
and on $\Z^k$ this is just the usual partial ordering.

For $\bs{p} = (p_1, p_2, \ldots, p_k) \in \Zcol^*$, if $p_i<p_{i+1}$, we say that $\bs{p}$ has an \emph{increase} at $i$, the \emph{increase set} of $\bs{p}$ is the set of increases of $\bs{p}$. (We use this terminology instead of the more standard ``ascent'' to avoid confusion with ascents of a permutation). We say that $\bs{p}$ is \emph{weakly-increasing} if $p_1\le p_2\le\cdots\le p_k$ and \emph{strictly-increasing} if $p_1<p_2<\cdots<p_k$. These notions are well-defined because different elements of the same equivalence class have the same relative orders. We define weakly and strictly increasing elements of $\Z^*$ similarly. Let $\ZWInc^k$ be the set of all weakly-increasing elements of $\Z^*$ of length $k$, and $\ZWInc^* = \bigcup_{k\ge 0} \ZWInc^k$.

Next, we define a generalization of the shuffle product, applied to colored words, which we call the \emph{colored shuffle product}. Similar constructions have been used by Assaf and Searles \cite{AssafSearles}, and by Nadeau and Tewari \cite{NadeauTewari}. For $\bs{p}, \bs{q} \in \Zcol^*$, let
\[
\bs{p}\bs{q} := \bs{p} \shuffle \bs{q} \uparrow \maxcol{\bs{p}},
\]
where $\uparrow$ associates more strongly than $\shuffle$. The color shift ensures that in every shuffled word appearing in $\bs{p}\bs{q}$, the colors of entries of $\bs{p}$ and the colors of entries of $\bs{q}$ are disjoint. If $\bs{p}$ and $\bs{q}$ are monocolored words, the product $\bs{p}\bs{q}$ will involve exactly two colors. Call the colored words appearing in this product \emph{colored shuffled words}. Note in particular that the colored shuffle product is positive--all the coefficient are either 0 or 1.

\begin{example}
    If $\bs{p}$ and $\bs{q}$ are monocolored, the product $\bs{p}\bs{q}$ will involve exactly two colors. For example, if $\bs{p} = 12, \bs{q} = 2$, then 
    \[
    \bs{p}\bs{q} = 1^{[1]}2^{[1]}2^{[2]} + 1^{[1]}2^{[2]}2^{[1]} +
    2^{[2]}1^{[1]}2^{[1]}.
    \]
\end{example}

\begin{proposition}
    Multiplication in the colored shuffle algebra is associative.
\end{proposition}
\begin{proof}
    \begin{align*}
        \bs{p}(\bs{q}\bs{r}) &= \bs{p} \shuffle (\bs{q} \shuffle \bs{r} \uparrow \maxcol{\bs{q})} \uparrow \maxcol{\bs{p}} \\&= (\bs{p} \shuffle \bs{q} \uparrow \maxcol{\bs{p})} \shuffle \bs{r} \uparrow (\maxcol{\bs{q}} + \maxcol{\bs{p})} \\&= (\bs{p}\bs{q})\bs{r},
    \end{align*}
    where we have used the associativity of the shuffle product, along with the fact that $(\bs{p} \shuffle \bs{q}) \uparrow a = \bs{p} \uparrow a \shuffle \bs{q} \uparrow a$.
\end{proof}

On the other hand, the colored shuffle product is not commutative; for example, $1^{[1]} \cdot 2^{[1]} = 1^{[1]}2^{[2]} + 2^{[2]}1^{[1]} \neq 1^{[2]}2^{[1]} + 2^{[1]}1^{[2]} = 2^{[1]} \cdot 1^{[1]}$.

We end this subsection with the definition of our main algebraic object:
\begin{definition}
The \textit{colored shuffle algebra} $\QAlg := \Q[\Zcol^*]$ is the $\Q$-algebra spanned by colored words, with multiplication given by the colored shuffle product extended linearly.
\end{definition}

Elements of $\QAlg$ are finite linear combinations of colored words. We will often use the notation $\bsvec{a}$ for an arbitrary element of $\QAlg$, to distinguish from the case of a single word.

\subsection{Compatible sequences}

Next, we define compatible sequences. These were first studied by Billey, Jockusch, and Stanley \cite{BilleyJockuschStanley}, working with (uncolored) reduced words of permutations. We note, however, that the usual definitions make sense in the context of arbitrary colored words, and we work in that generality. In the case of uncolored reduced words, these definitions match \cite{BilleyJockuschStanley}.

\begin{definition} \label{def:compatible-sequences}
A compatible sequence is a two-row array of the form:

\[\bs{\alpha} := \begin{pmatrix} \bs{t} \\ \bs{b} \end{pmatrix} =  \begin{pmatrix} t_1 & t_2 & \cdots & t_k \\ b_1 & b_2 & \cdots & b_k \end{pmatrix}, \qquad \bs{t} \in \Zcol^k,\bs{b}\in \Z^k,\] such that

\begin{equation} \label{eq:comp-seq}
\bs{b}\in \ZWInc^k, \qquad\qquad \bs{b}\le \bs{t}, \qquad\qquad
 t_j < t_{j+1} \implies b_j < b_{j+1}.
\end{equation}
\end{definition}

We call $\bs{t}$ the \emph{top row} of $\alpha$, and $\bs{b}$ the \emph{bottom row}. For notational convenience, we may also write $\bs{\alpha} = (\bs{t}, \bs{b})$. We call $\bs{\alpha}$ \emph{positive} if all of its entries are positive. 

For $\bs{t}\in \Zcol^*$, let $\mc{C}(\bs{t})$ be the set of compatible sequences with top row $\bs{t}$ and let $\BR(\bs{t})\subseteq \ZWInc^k$ be the set of bottom rows of elements of $\mc{C}(\bs{t})$.

The set $\BR(\bs{t})$ has a nice structure. Define $m(\bs{t})\in\ZWInc^k$ by the following recursive formula: \begin{equation} \label{eq:max-bot-row} 
m(\bs{t})_k = \val(t_k), \qquad m(\bs{t})_j = \begin{cases} m(\bs{t})_{j+1}, & \text{if } t_j \ge t_{j+1}; \\ \min(\val(t_j), m(\bs{t})_{j+1} - 1), & \text{if } t_j < t_{j+1}.\end{cases}\qquad \text{for all $j<k$}.
\end{equation}

It is always the case that $m(\bs{t})\le \bs{t}$. If $\bs{t}$ is weakly increasing, then $m(\bs{t})=\val(\bs{t})$, so the map $m(\rule{6pt}{.5pt}):\Zcol^*\to\ZWInc^*$ is surjective.

\begin{lemma} \label{lem:possible-bottom-rows}
$m(\bs{t})$ is a maximal element of $\BR(\bs{t})$ under the partial ordering on $\Z^k$. Furthermore,
\[\BR(\bs{t}) = \BR(m(\bs{t})) = \{\bs{b}\in \ZWInc^k \;|\; \bs{b}\le m(\bs{t}) \text{ and for all $j$, $b_j < b_{j+1}$ whenever $t_j < t_{j+1}$} \}.\]
In particular, $\BR(\bs{t})$ only depends on $m(\bs{t})$.
\end{lemma}

\begin{proof}
It is an easy check that $m(\bs{t})$ satisfies \eqref{eq:comp-seq}, and so it is indeed a bottom row for $\bs{t}$. Let $\bs{b}\in\ZWInc^k$ satisfy $t_j<t_{j+1}\implies b_j<b_{j+1}$. If $\bs{b}\le \bs{t}$, we need to show that $\bs{b}\le m(\bs{t})$ i.e. that $b_j\le m(\bs{t})_j$ for all $j$. If $j=k$, this follows from the definition, so assume $j<k$ and $b_{j+1}\le m(\bs{t})_{j+1}$. Then either $t_j\ge t_{j+1}$, in which case $b_j\le b_{j+1}\le m(\bs{t})_{j+1} = m(\bs{t})_j$, or $t_j<t_{j+1}$, in which case $b_j \le b_{j+1}-1 \le m(\bs{t})_{j+1}-1$ and $b_j\le t_j$, so $b_j\le \min(\val(t_j), m(\bs{t})_{j+1} - 1) = m(\bs{t})_j$.

By \eqref{eq:max-bot-row}, $\bs{t}$ and $m(\bs{t})$ have increases at exactly the same indices, so \[\BR(\bs{t}) = \{\bs{b}\in \ZWInc^k \;|\; \bs{b}\le m(\bs{t}) \text{ and for all $j$, $b_j < b_{j+1}$ whenever $m(\bs{t})_j < m(\bs{t})_{j+1}$} \}.\qedhere\]
\end{proof}

\begin{remark}
It is possible that $\BR(\bs{t}) = \BR(\bs{t'})$ when $\bs{t}\ne \bs{t'}$. This can happen even if $\bs{t}$ and $\bs{t'}$ are reduced words for the same permutation. For instance, $s_6s_8s_4s_1 = s_4s_8s_6s_1 \ne s_4s_9s_6s_1\in S_\Z$ and $\BR(6^{[1]}8^{[1]}4^{[1]}1^{[1]}) = \BR(4^{[1]}8^{[1]}6^{[1]}1^{[1]}) = \BR(4^{[1]}9^{[1]}6^{[1]}1^{[1]})$: all three words have the same maximal element $0111$.
\end{remark}

We now define \emph{bottom-row-equivalence}. For $\bs{p}, \bs{q} \in \Zcol^*$, we say that
\begin{equation}
    \bs{p} \equiv \bs{q} \text{ if and only if } m(\bs{p}) = m(\bs{q}).
\end{equation}

We record the following consequence of the proof of Lemma \ref{lem:possible-bottom-rows}

\begin{lemma} \label{lem:max-bottom-inc}
    For any word $\bs{p}$ and any $j < k$, $m(\bs{p})_j < m(\bs{p})_{j+1}$ if and only if $\bs{p}_j < \bs{p}_{j+1}$.
\end{lemma}

\subsection{Back-symmetric and back-quasisymmetric functions}

Let $\bs{x} = (\ldots,x_{-2},x_{-1},x_0,x_1,x_2,\ldots)$ be a doubly-infinite sequence of variables. For $A\subseteq\Z$, let $\bs{x}_A = \{x_i\;|\;i\in A\}$ and consider the ring \[\Poly_A := \Q[\bs{x}_A]\] of polynomials in $\bs{x}_A$ with rational coefficients. This is a graded ring: $\Poly_A = \bigoplus_{k\ge 0} \Poly^k_A$, where $\Poly^k_A$ is the set of degree-$k$ homogeneous polynomials in $\Poly_A$.

If $\bs{p}\in \Z^k$ has $\supp\bs{p}\subseteq A$, we have the monomial
\begin{equation} \label{eq:monomial-defn}
\bs{x}^{\bs{p}} := x_{\bs{p}_1}x_{\bs{p}_2}\cdots x_{\bs{p}_k}\in \Poly^k_A,
\end{equation}
and the set $\{\bs{x}^{\bs{p}} \;|\; \bs{p}\in\ZWInc^k, \supp\bs{p}\subseteq A\}$ is linearly independent, and forms a basis for $\Poly^k_A$ when $A$ is finite. Write $\supp \bs{x}^{\bs{p}} = \supp\bs{p}$, and let $\supp f$ be the union of the supports of its monomials.

\begin{remark}
    Note that \eqref{eq:monomial-defn} is slightly nonstandard notation; usually one has a composition $\beta = (\beta_1,\beta_2,\ldots)$ and writes $\bs{x}^{\bs{\beta}} = x_1^{\beta_1} x_2^{\beta_2}\cdots$. To translate to this setting, simply let $\beta_i$ be the number of entries of $\bs{p}$ which equal $i$.)
\end{remark}

If $B\subseteq A$, let $\rho_{A,B}:\Poly_A\to \Poly_B$ be the map that sends $x_i, i\in B$ to itself and all other variables to 0. $\rho_{A,B}$ is a map of graded rings, and acts on monomials via \[\bs{x}^{\bs{p}}\mapsto \begin{cases} \bs{x}^{\bs{p}}, & \text{if }\supp \bs{p}\subseteq B, \\ 0, & \text{otherwise}.\end{cases}\]

$f\in\Poly_A$ is \emph{symmetric} if it is invariant under the action of $S_A$ permuting the variables in $\bs{x}_A$, and \emph{quasisymmetric} if the coefficients for $\bs{x}^{\bs{p}}$ and $\bs{x}^{\bs{q}}$ are equal for any $\bs{p},\bs{q}\in\ZWInc^k$ which have the same increase sets. For finite $A$, let $\Lambda_A\subseteq \Poly_A$ be the ring of symmetric polynomials on $A$, and let $\QSym_A\subseteq \Poly_A$ be the ring of quasisymmetric polynomials. These rings are also graded by degree:
\[\Lambda_A = \bigoplus_{k\ge 0} \Lambda_A^k, \qquad\qquad \QSym_A = \bigoplus_{k\ge 0} \QSym_A^k.\]
The (quasi)symmetry does allow us to fruitfully take limits: if $A$ is infinite, let
\[\Lambda_A = \varprojlim \Lambda_B, \qquad\qquad \QSym_A = \varprojlim \QSym_B,\]
where in both cases the projective limit is taken over all finite subsets $B\subseteq A$ relative to the maps $\rho_{A,B}$. These are the rings of \emph{(quasi)symmetric functions} \cite{Macdonald-book, Gessel-quasisymmetric}.

\medskip

Next, we move to the back-stable setting. One must be careful since the set of all power series, even with restrictions on degree, is not an inverse limit of polynomial rings. Fortunately, the relevant functions lie in a much smaller ring. Let $R$ be the ring of formal power series $f$ in $\bs{x}$ which have bounded total degree and support bounded above. The former condition means that there exists an $M$ such that all monomials in $f$ have total degree $\le M$, and the latter condition means that there exists an $N$ such that $\supp f\subseteq (-\infty,N]$. $R$ is also a graded ring: $R = \bigoplus_{k\ge 0} R^k$, where $R^k$ is the subring of $R$ consisting of elements which are homogeneous of degree $k$.

For $b\in\Z$, a function $f\in R$ is \emph{back $b$-(quasi)symmetric} if under any specialization $x_j\mapsto z_j\in\Q, j>b$, the resulting function is (quasi)symmetric. If $f$ is back $b$-(quasi)symmetric, then it is back $b'$-(quasi)symmetric for any $b'<b$. We also use the terminology that a back $b$-(quasi)symmetric function is (quasi)symmetric in $\bs{x}_{(-\infty,b]}$.

$f$ is \emph{back-(quasi)symmetric} if it is back $b$-(quasi)symmetric for any $b\in\Z$. Let $\ol{R}$ be the ring of back-symmetric functions and $\ol{Q}$ be the ring of back-quasisymmetric functions.

Given subrings $R',R''$ of $R$, let $R'\otimes R''$ be the subring of $R$ generated by products $r'r''$, $r'\in R', r''\in R''$. The set of back $b$-symmetric functions form the ring $\Lambda_{(-\infty,b)} \otimes \Poly_{[b,\infty)}$ and the set of back $b$-quasisymmetric functions form the ring $\QSym_{(-\infty,b)} \otimes \Poly_{[b,\infty)}$. Explicitly this means the following. Any $\bs{p}\in\ZWInc^*$ can be uniquely written $\bs{p} = \bs{q}\circ\bs{r}$, where $\supp\bs{q}\subseteq (-\infty,b]$ and $\supp\bs{r}\subseteq (b,\infty)$. $f$ is back $b$-quasisymmetric if and only if the coefficient of $\bs{x}^{\bs{p}}$ in $f$ equals the coefficient of $\bs{x}^{\bs{q'}\circ\bs{r}}$, where $\bs{q'}$ is any word in $\ZWInc^*$ with the same length and increase set as $\bs{q}$ such that $\supp\bs{q'}\subseteq (-\infty,b]$. Back $b$-symmetry has a similar definition, where the condition on increase set is replaced by the condition that there exists $w\in S_\Z$ that sends $\bs{x}^{\bs{q}}$ to $\bs{x}^{\bs{q'}}$. Equivalently, $f$ is back $b$-symmetric if and only if $s_if = f$ for all $i<b$.

Taking the union over all $b\in\Z$, we have
\begin{equation*}
\ol{R} = \bigcup_{b\in\Z} \Lambda_{(-\infty,b]} \otimes \Poly_{(b,\infty)} \qquad \text{and} \qquad \ol{Q} = \bigcup_{b\in\Z}\QSym_{(-\infty,b]} \otimes \Poly_{(b,\infty)}.
\end{equation*}
In particular, since $\Lambda_{(-\infty,b]}\subseteq \QSym_{(-\infty,b]}$, $\ol{R}\subseteq\ol{Q}$.

$\ol{R}$ and $\ol{Q}$ can also be expressed in a more straightforward way:

\begin{proposition}[{\!\!\cite{LamLeeShimiozono-schubert, NadeauTewari}}] \label{prop:back-quasi}
For any $b\in\Z$
\[\ol{R} = \Lambda_{(-\infty,b]} \otimes \Poly_\Z \qquad \text{and} \qquad \ol{Q} = \QSym_{(-\infty,b]} \otimes \Poly_\Z.\]
\end{proposition}

\subsection{Back-stable Schubert and slide polynomials}

Next, we will define \emph{back-stable slide polynomials}, a set of functions associated to the compatible sequences with a fixed top row. Slide polynomials (in their non-back-stable form) first appeared in the  Billey-Jockusch-Stanley formula \eqref{eq:bjs-formula}, and were systematically studied by Assaf and Searles \cite{AssafSearles}. In their back-stable form, they were defined by Nadeau and Tewari \cite{NadeauTewari}.

As with compatible sequences, we will associate a back-stable slide polynomial to any colored word (in fact, any element of $\QAlg$). Using colored words does not introduce any new functions, but is useful in that it allows us to study the back-stable slide decomposition of the Schubert product using the combinatorics of shuffled colored words.

For any compatible sequence $\bs{\alpha} = (\bs{t}, \bs{b})$, let $\bs{x}^{\bs{\alpha}} = \bs{x}^{\bs{b}}$.

\begin{definition}
Let $\bs{p}\in\Zcol^*$. The \emph{back-stable slide polynomial} associated to $\bs{p}$ is the function \[\olF_{\bs{p}} = \sum_{\bs{\alpha}\in\mc{C}(\bs{p})} \bs{x}^{\bs{\alpha}}\in \ol{Q}.\]
\end{definition}

By Lemma \ref{lem:possible-bottom-rows}, $\olF_{\bs{p}} = \olF_{m(\bs{p})}$, and so $\olF_{\bs{p}} = \olF_{\bs{q}}$ if and only if $m(\bs{p}) = m(\bs{q})$. Since every element of $\BR(\bs{p})$ is the bottom row for exactly one compatible sequence with top row $\bs{p}$, we can also write 
\begin{equation*} \olF_{\bs{p}} = \sum_{\bs{b}\in\BR(\bs{p})} \bs{x}^{\bs{b}} = \bs{x}^{m(\bs{p})} + \text{lower terms},
\end{equation*}
where ``lower terms'' refers to the partial ordering on words.

The $\olF_{\bs{p}}$ are homogeneous of degree $\ell(\bs{p})$, and have support contained in $(-\infty, \bs{p}_{\ell(\bs{p})}]$; thus they are elements of $R$. Every coefficient is either 0 or 1. 

\begin{lemma}[{\!\!\cite[Lemma~4.6, Proposition~4.8]{NadeauTewari}}] \label{lem:back-quasi-sym}
$\olF_{\bs{p}}\in \ol{Q}$ for all $\bs{p}$. Furthermore, if $\bs{p}\in\ZWInc^k$ with $\bs{p}_1=a$ such that the increase set of $\bs{p}$ has size $d$, then $\olF_{\bs{i}}$ is back $(a+d)$-quasisymmetric.
\end{lemma}

\begin{remark} \label{rem:slide-comparison}
We have indexed the polynomials $\olF_{\bs{p}}$ by words, whereas Assaf and Searles \cite{AssafSearles} primarily use compositions. Nadeau and Tewari \cite{NadeauTewari} use both notations. In the case of increasing words, these notations are equivalent: the index-$i$ entry of a composition denotes the number of entries of the corresponding word that equal $i$.

We find the word approach more useful for three reasons. The first is that using words provides us with more flexibility: the same function can be the slide polynomial of multiple words, and we sometimes want to distinguish these instances. The second reason is that the product formula is simpler when using words instead of compositions, involving the shuffle product directly (Proposition~\ref{prop:word-poly-multiplication}). The third and most important reason is that the action of the differential operators in the next section is more naturally expressed on words. 
\end{remark}

Next, we define the \emph{back-stable Schubert polynomials}. These functions have a long folk-lore history and were studied systematically by Lam, Lee, and Shimozono \cite{LamLeeShimiozono-schubert}. Ideas of Li \cite{Li-back-stable-conjecture} also suggest the usefulness of working in this context.

\begin{definition}
Fix $w\in S_\Z$. The \emph{back-stable Schubert polynomial} $\ol{\Sch}_w$ associated to $w$ is given by 
\begin{equation} \label{eq:Schubert-monomial-expansion}
\ol{\Sch}_w = \sum_{\bs{w}\in \RW(w)} \sum_{\bs{p}\in\BR(\bs{w})} \bs{x}^{\bs{p}} \in \ol{R}.
\end{equation}
\end{definition}

The next proposition follows directly from the definitions.

\begin{proposition}
\begin{equation} \label{eq:Schubert-word-expansion}
\ol{\Sch}_w = \sum_{\bs{w}\in \RW(w)} \olF_{\bs{w}}.
\end{equation}
\end{proposition}

\begin{theorem}[\!\!\cite{LamLeeShimiozono-schubert, NadeauTewari}] \label{thm:word-basis} 
The set $\{\ol{\Sch}_w \;|\; w\in S_\Z\}$ is a $\Q$-basis of $\ol{R}$, while the set $\{\olF_{\bs{p}} \;|\; \bs{p}\in \ZWInc^*\}$ is a $\Q$-basis of $\ol{Q}$.
\end{theorem}

\section{The algebra of words} \label{sec:colored-shuffle-algebra}

In this section, we connect the colored shuffle algebra and the ring of back-stable quasisymmetric functions by showing that the latter is a quotient of the former. Moreover, we discuss two operators, $\xi$ and $\nabla$, which act on $\QAlg$, and descend to previously-studied operators on $\ol{Q}$ and $\ol{R}$, respectively.

\subsection{Homomorphism between $\QAlg$ and rings of polynomials}

We succinctly represent linear combinations of slide polynomials by allowing 
$\olF$ to be indexed by elements of $\QAlg$.
\begin{definition} \label{def:bs-slide-vector}
    For $\bsvec{a} = \sum_{\bs{p} \in \Zcol^*} c_{\bs{p}}\bs{p} \in \QAlg$, we define $\olF_{\bsvec{a}} = \sum_{\bs{p} \in \Zcol^*} c_{\bs{p}}\olF_{\bs{p}}$.
\end{definition}

Nadeau and Tewari used slightly different notation, but their proof of the following multiplication rule goes through:
\begin{proposition}[{\!\!\cite[Proposition~4.12]{NadeauTewari}}] \label{prop:word-poly-multiplication}    $\olF_{\bsvec{a}}\olF_{\bsvec{b}} = \olF_{\bsvec{a} \bsvec{b}}$.
\end{proposition}

Consider the linear map $\phi : \QAlg \rightarrow \ol{Q}$ defined by $\bsvec{a} \mapsto \olF_{\bsvec{a}}$. For notational convenience, we will also use $\phi$ to refer to the map it induces on any subquotient of $\QAlg$. Recall that colored words $\bs{p}$ and $\bs{q}$ are bottom-row-equivalent, $\bs{p}\equiv\bs{q}$, if $m(\bs{p}) = m(\bs{q})$.

\begin{proposition}
$\phi$ is a surjective ring homomorphism, with kernel spanned by the formal differences $\bs{p}-\bs{q}$, where $\bs{p}\equiv\bs{q}$.
\end{proposition}

\begin{proof}
$\phi$ is a ring homomorphism by Definition~\ref{def:bs-slide-vector} and Proposition~\ref{prop:word-poly-multiplication}, and a surjective linear map by Theorem~\ref{thm:word-basis}. Since basis vectors $\bs{p}\in\Zcol^*$ are mapped to basis vectors $\olF_{\bs{p}}$, $\ker\phi$ is spanned by formal differences $\bs{p}-\bs{q}$, where $\phi(\bs{p})=\phi(\bs{q})$. This happens precisely when $\bs{p}\equiv \bs{q}$, since $\phi(\bs{p}) = \olF_{\bs{p}} = \olF_{m(\bs{p})}$.
\end{proof}
Therefore, we have realized the ring of quasisymmetric functions as a quotient of the colored shuffle algebra:
\begin{corollary} \label{cor:colored-isomorphism}
    $\QAlg/ \!\!\equiv\; \cong \ol{Q}$ as $\Q$-algebras.
\end{corollary}

Next, we want to realize $\ol{R}$ in $\QAlg$. We will do so as both a subspace and a quotient algebra. For $w\in S_\Z$, the \emph{Schubert vector} associated to $w$ is
\[
\Schvec(w) := \sum_{\bs{w} \in \RW(w)} \bs{w} \;\in \QAlg.
\]
By \eqref{eq:Schubert-word-expansion},
\begin{equation} \label{eq:word-poly-of-Schubert-vector}
\olF_{\Schvec(w)} = \sum_{\bs{w} \in \RW(w)} \olF_{\bs{w}} = \ol{\Sch}_w.
\end{equation}
Let $\mcS$ be the $\Q$-linear span of $\{\Schvec(w) \;|\; w \in S_\Z\}$, and let $\mcR$ be the subalgebra of $\QAlg$ generated by $\mcS$. Products of the form $\Schvec(w_1)\cdots \Schvec(w_k), w_i\in S_\Z$ form a basis for $\mcR$.

\begin{proposition} \label{prop:S-R-bijection}
    $\phi : \mcS \rightarrow \ol{R}$ is a bijective linear map.
\end{proposition}

\begin{proof}
    By Theorem~\ref{thm:word-basis}, $\{\ol{\Sch}_w \;|\; w \in S_\Z\}$ forms a basis for $\ol{R}$. By \eqref{eq:word-poly-of-Schubert-vector}, $\phi(\Schvec(w)) = \ol{\Sch}_w$, so $\phi$ is a bijection $\mcS\to\ol{R}$.
\end{proof}

\begin{proposition}
    $\phi : \mcR / \!\!\equiv\; \rightarrow \ol{R}$ is an isomorphism.
\end{proposition}
\begin{proof}
    Since $\mcR$ is a subalgebra of $\QAlg$, by Corollary \ref{cor:colored-isomorphism}, $\phi$ maps $\mcR / \!\!\equiv\;$ isomorphically onto its image. By the previous proposition, $\phi(\mcR / \!\!\equiv) = \phi(\mcR) = \phi(\mcS) = \ol{R}$, since $\mcS$ generates $\mcR$ and $\ol{R}$ is closed multiplicatively.
\end{proof}

\begin{corollary}
    For all $\bsvec{a} \in \mcR$, there exists a unique $\bsvec{b} \in \mcS$ such that $\bsvec{a} \equiv \bsvec{b}$.
\end{corollary}
\begin{proof}
    This follows from Proposition \ref{prop:S-R-bijection}, since there is a unique $\bsvec{b}\in \mcS$ such that $\phi(\bsvec{b}) = \phi(\bsvec{a})$.
\end{proof}

Denote this unique $\bsvec{b}$ as $\mcS(\bsvec{a})$.

\begin{corollary} \label{cor:S-Schub-prod-is-expansion}
    $\mcS(\Schvec(u) \Schvec(v)) = \sum_w \ol{c_{u,v}^w} \Schvec(w)$.
\end{corollary}
\begin{proof}
    \[
    \olF_{\Schvec(u)\Schvec(v)} = \olF_{\Schvec(u)}\olF_{\Schvec(v)} = \ol{\Sch}_u\ol{\Sch}_v = \sum_w \ol{c_{u,v}^w} \ol{\Sch}_w = \olF_{\sum_w \ol{c_{u,v}^w} \Schvec(w)},
    \]
    so we have
    \begin{equation} \label{eq:Schub-word-equivalence}
    \Schvec(u)\Schvec(v) \equiv \sum_w \ol{c_{u,v}^w} \Schvec(w),
    \end{equation} and the right side is in $\mcS$.
\end{proof}

The maps in this subsection are represented by the following commutative diagram:

\begin{center}
\begin{tikzcd}
    \mcS \arrow[r,hook] \arrow[dr, hookrightarrow, two heads, "\phi", swap] & \mc{R} \arrow[l, bend right, two heads, "\exists ! \mcS(\_)", swap] \arrow[r, hook] \arrow[d, two heads, "\phi"] & \QAlg \arrow[d, two heads, "\phi"] \\
    &\ol{R} \arrow[r, hook] & \ol{Q} &
\end{tikzcd}
\end{center}

\subsection{Weight-preserving maps on elements of $\QAlg$}
Denote $\N[\Zcol^*] \subseteq \QAlg$ the set of $\N$-linear combinations of elements of $\Zcol^*$. Any element $\sum_{\bs{p} \in \Zcol^*} c_{\bs{p}} \bs{p} \in \N[\Zcol^*]$ can be thought of a multiset where each $\bs{p}$ has multiplicity $c_{\bs{p}}$. For elements of $\N[\Zcol^*]$ we sometimes use multiset operations such as $\in$ for membership queries or $|\cdot|$ for multiset order without comment. Under this convention, $\Schvec(w) = \RW(w)$, but we maintain the two different notations for cases where we want to make the difference explicit.

For $\bs{p} \in \ZWInc^*$, consider $\Q[\phi^{-1}(\olF_{\bs{p}})]$, the subspace of $\QAlg$ spanned by all $\bs{q}\in\Zcol^*$ with $m(\bs{q}) = \bs{p}$. Let $P_{\bs{p}}:\QAlg\to \Q[\phi^{-1}(\olF_{\bs{p}})]$ be the projection operator onto this space, and let $T_{\bs{p}}\bsvec{a}$ be the sum of the coefficients of $P_{\bs{p}}\bsvec{a}$.
Concretely, if $\bsvec{a} = \sum_{\bs{q} \in \Zcol^*}c_{\bs{q}}\bs{q}$, then

\begin{equation} \label{eq:projection-op-def}
    P_{\bs{p}}\bsvec{a} = \sum_{\substack{\bs{q} \in \Zcol^*\\ m(\bs{q}) = \bs{p}}}c_{\bs{q}}\bs{q}, 
    \qquad \text{and}\qquad T_{\bs{p}}
    \bsvec{a} = \sum_{\substack{\bs{q} \in \Zcol^*\\ m(\bs{q}) = \bs{p}}}c_{\bs{q}}.
\end{equation}

For any $\bsvec{a}, \bsvec{b} \in \N[\Zcol^*]$, it is the case that

\begin{equation*}
    \bsvec{a} \equiv \bsvec{b} \qquad\text{ if and only if } \qquad \forall \bs{p} \in \ZWInc^*, \;\; T_{\bs{p}}\bsvec{a} = T_{\bs{p}}\bsvec{b},
\end{equation*}

Then, given $\bsvec{a} \equiv \bsvec{b}$, there must exist some multiset bijection $\psi|_{\bs{p}}$ between $P_{\bs{p}}\bsvec{a}$ and  $P_{\bs{p}}\bsvec{b}$. Furthermore, we can take the union of $\psi|_{\bs{p}}$ for all $\bs{p}$ to get a bijection $\psi$ between $\bsvec{a}$ and $\bsvec{b}$ which is maximal bottom row preserving: $m(\psi(\bs{q})) = m(\bs{q})$.

Applying this discussion to the Schubert product, we get the following proposition:
\begin{proposition} \label{prop:multiset-bijection}
    For any $u,v$, there must exist some multiset bijection between $\Schvec(u)\Schvec(v)$ and $\sum_w \ol{c_{u,v}^w} \Schvec(w)$ which preserves maximal bottom row.
\end{proposition}
\begin{proof}
$\Schvec(u)\Schvec(v) \in \N[\Zcol^*]$, and by Schubert structure constant positivity, $\sum_w \ol{c_{u,v}^w} \Schvec(w) \in \N[\Zcol^*]$. By \eqref{eq:Schub-word-equivalence}, the two multisets are bottom-row equivalent, and the above paragraph gives the desired map $\psi$.
\end{proof}

This allows for a quick proof of the following:
\begin{proposition}[Nenashev] \label{prop:nenashev}
For all $u,v$, \[\binom{\ell(u) + \ell(v)}{\ell(v)} |\RW(u)| |\RW(v)| = \sum_w \ol{c_{u,v}^w} |\RW(w)|\]
\end{proposition}
\begin{proof}
    There must exist some bijection $\psi : \Schvec(u)\Schvec(v) \rightarrow \sum_w \ol{c_{u,v}^w} \Schvec(w)$, so we have $\binom{\ell(u) + \ell(v)}{\ell(v)} |\RW(u)| |\RW(v)| = |\Schvec(u)\Schvec(v)| = |\sum_w \ol{c_{u,v}^w} \Schvec(w)| = \sum_w \ol{c_{u,v}^w} |\RW(w)|$.\qedhere
\end{proof}

The tensor product $\mcS\otimes\mcS$ is the vector subspace of $\QAlg$ spanned by products $\Schvec(u)\Schvec(v)$. As a multiset, $\Schvec(u)\Schvec(v)$ consists of all colored shuffles of reduced words of $u$ and $v$. Let $\M := \bigcup_{u,v} \Schvec(u)\Schvec(v)$ be the set of all such shuffles for all pairs of permutations $u,v\in S_\Z$. Then the linear span $\Q[\M]$ is a subspace of $\QAlg$ strictly containing $\mcS\otimes\mcS$. We claim:

\begin{proposition} \label{prop:magic-map-exists}
There exists a map $\psi:\M\to\RW$ whose linear extension makes the following diagram commute:
\begin{center}
\begin{tikzcd}
    \mcS \otimes \mcS \arrow[r, "\mcS(\_)"] \arrow[d, hook] & \mcS \arrow[d, hook] &  \\
    \Q[\M] \arrow[r, dashed, "\psi"] \arrow[rd, two heads, swap, "\phi"] & \Q[\RW] \arrow[d, two heads, "\phi"] \\
    & \ol{Q} &
\end{tikzcd}
\end{center}
\end{proposition}

\begin{proof}
$\mcS\otimes\mcS$ has basis given by products $\Schvec(u)\Schvec(v)$, and following either $\mcS\otimes\mcS\to \Q[\M]\to \ol{Q}$ or $\mcS\otimes\mcS\to \mcS \to \Q[\RW] \to \ol{Q}$ sends $\Schvec(u)\Schvec(v)$ to the product $\ol{\Sch}_u\ol{\Sch}_v$.

Consider an element of $\M$; this is a colored shuffled word $\bs{p}\in\Schvec(u) \Schvec(v)$ for some $u,v\in S_\Z$. Both $\Schvec(u)\Schvec(v)$ and $\mcS(\Schvec(u)\Schvec(v))$ are elements of $\N[\Zcol^*]$, so they can be considered as multisets. By Corollary \ref{cor:S-Schub-prod-is-expansion} and Proposition \ref{prop:multiset-bijection}, there exists a bijection from $\Schvec(u)\Schvec(v)$ to $\mcS(\Schvec(u)\Schvec(v))$ preserving maximal bottom row, and any such bijection over all pairs $u$ and $v$ is the desired map.
\end{proof}

\begin{problem} \label{problem:Schubert}
Find an explicit map $\psi$ satisfying Proposition \ref{prop:magic-map-exists}.
\end{problem}

A solution to Problem \ref{problem:Schubert} would give a combinatorial proof of Schubert structure positivity since for any reduced word $\bs{w}$ for $w$,
\begin{equation} \label{eq:magic-map-LR-coeffs}
\ol{c_{u,v}^w} = |\psi^{-1}(\bs{w}) \cap \Schvec(u)\Schvec(v)|.
\end{equation}

Such a map is (unsurprisingly) hard to find, but we are able to prove that any such map must have certain properties. One such property is that $\psi$ will always fix the value of the final entry: for $\bs{p} \in \Zcol^k$, $\psi(\bs{p})_k = \val(\bs{p}_k)$. In fact, more can be said about how $\psi$ acts on the last few entries of $\bs{p}$ (Proposition~\ref{prop:psi-maximal-inc-suf}).

In addition, we conjecture:
\begin{conjecture} \label{conj:magic-map}
    There exists some $\psi$ satisfying Proposition \ref{prop:magic-map-exists} such that $\psi \xi = \xi \psi$, where $\xi$ is the operator defined in the next subsection \eqref{eq:nabla-xi-word-def} which removes the first letter of a word.
\end{conjecture}

If Conjecture \ref{conj:magic-map} holds, it would open the possibility of building up $\psi$ inductively: knowing where $\psi$ sends to length-$n$ words would severely constrain where it can send length-$(n+1)$ words.

\begin{remark}
One possible approach to Problem~\ref{problem:Schubert} would be to look for bijections satisfying Proposition~\ref{prop:magic-map-exists} in the context of Pieri's rule; that is, a maximal-bottom-row-preserving bijection between $\Schvec(u)\Schvec(s_{i+j} \ldots s_{i+1} s_{i})$ and $\mcS(\Schvec(u)\Schvec(s_{i+j} \ldots s_{i+1} s_{i}))$. Sjoblom \cite{sjoblom2024some} gives a bijection between these sets; however, his bijection does not preserve maximal bottom row.

\end{remark}
\begin{problem}
    Modify Sjoblom's bijection to preserve maximal bottom row.
\end{problem}

\subsection{Differential operators on Schubert and slide polynomials}

Next, we discuss some linear operators which act on $\QAlg$.

\begin{proposition} \label{prop:leibniz}
    Any linear operator $\bigstar$ which acts on $(p_1, \ldots, p_k) \in \Zcol^*$ by
    \[\bigstar (p_1, \ldots, p_k) = f(\val(p_1)) \cdot (p_2, \ldots, p_k),
    \]
    where $f$ is some function $\Z \rightarrow \Q$, satisfies the Leibniz rule up to maximal bottom row equivalence (ie $\bigstar (\bs{p} \bs{q}) \equiv \bigstar(\bs{p})\bs{q} + \bs{p}\!\bigstar\!(\bs{q})$).
\end{proposition} 
\begin{proof}
    The statement is trivially true if $\bs{p} = \emptyset$ or if $\bs{q} = \emptyset$. Assuming otherwise, let $\bs{p'} = (p_2, \ldots, p_{\ell(\bs{p})})$ and $\bs{q'} = (q_2, \ldots, q_{\ell(\bs{q}})$. Note that $\maxcol{\bs{p}} \geq \maxcol{\bs{p'}}$ so $\bs{p'}\bs{q} = \bs{p'} \shuffle \bs{q} \uparrow \maxcol{\bs{p'}} \equiv \bs{p'} \shuffle \bs{q} \uparrow \maxcol{\bs{p}}$. So, $\bigstar(\bs{p} \bs{q}) = \bigstar(\bs{p} \shuffle \bs{q} \uparrow \maxcol{\bs{p}}) = \bigstar(p_1 \circ  (\bs{p'} \shuffle \bs{q} \uparrow \maxcol{\bs{p}}) + q_1 \circ  (\bs{p} \shuffle \bs{q'} \uparrow \maxcol{\bs{p}})) \equiv f(\val(p_1))\bs{p'}\bs{q} + f(\val(q_1))\bs{p}\bs{q'} = \bigstar(\bs{p})\bs{q} + \bs{p}\!\bigstar\!(\bs{q})$.
\end{proof}

\begin{remark}
A stronger statement is true. Call $\bs{p}\in \Zcol^*$ and $\bs{q}\in \Zcol^*$ \emph{color-equivalent}, $\bs{p} \ceq \bs{q}$, if an order-preserving reindexing of the colors sends $\bs{p}$ to $\bs{q}$. For example, $2^{[2]}1^{[2]}3^{[1]}1^{[3]} \ceq 2^{[5]}1^{[5]}3^{[2]}1^{[6]}$ since all values agree, and the colors in both words have the same relative ordering.

Color-equivalent colored words are always bottom-row-equivalent, and the operator $\bigstar$ from the previous proposition satisfies
\[
\bigstar (\bs{p} \bs{q}) \ceq \bigstar(\bs{p})\bs{q} + \bs{p}\!\bigstar\!(\bs{q}).
\]
\end{remark}

\begin{definition} \label{def:nabla-xi-word}
Let $\nabla$ and $\xi$ be the linear operators that act on $\Zcol^*$ by 
\begin{equation} \label{eq:nabla-xi-word-def}
\nabla (p_1, \ldots, p_k) = \val(p_1) (p_2, \ldots, p_k), \qquad\qquad
\xi (p_1, \ldots, p_k) = (p_2, \ldots, p_k).
\end{equation}
\end{definition}

By Proposition \ref{prop:leibniz}, $\nabla$ and $\xi$ satisfy the Leibniz rule up to maximal bottom row equivalence.

\begin{proposition} \label{prop:diff-Schub-action}
    \begin{equation} \label{eq:nabla-xi-Q-action}
        \nabla \Schvec(w) = \sum_{k, \; \ell(s_kw)<\ell(w)} k\Schvec(s_kw) \quad\text{ and }\quad \xi \Schvec(w) = \sum_{k, \; \ell(s_kw)<\ell(w)} \Schvec(s_kw).
    \end{equation}
\end{proposition}

\begin{proof}
    This can be seen because $\nabla \Schvec(w) = \sum_{\bs{w} \in \RW(w)} \bs{w}_1(\bs{w}_2, \ldots, \bs{w}_k)$. Note that given $s_k w < w$, reduced words for $s_k w$ are in bijection with reduced words for $w$ that start with $k$. Thus, $\sum_{\bs{w} \in \RW(w)} \bs{w}_1(\bs{w}_2, \ldots, \bs{w}_k) = \sum_{k, s_kw<w} k\Schvec(s_kw)$. Similar reasoning holds for $\xi$.
\end{proof}

Thus, we get the following corollary from the definition of $\mcR$:
\begin{corollary} \label{cor:nabla-xi-R}
    $\bsvec{a} \in \mcR$ implies that $\nabla \bsvec{a} \in \mcR$ and $\xi \bsvec{a} \in \mcR$.
\end{corollary}

Next, we show that $\nabla$ and $\xi$ descend to the quotient $\ol{R}$. Applying $\phi$ to \eqref{eq:nabla-xi-Q-action}, we are led to define the following operators on $\ol{R}$:
\begin{equation} \label{eq:nabla-xi-olR-action}
\nabla \ol{\Sch}_w = \sum_{k, \ell(s_kw)<\ell(w)} k\ol{\Sch}_{s_kw},
\qquad \text{and} \qquad
\xi \ol{\Sch}_w = \sum_{k, \ell(s_kw)<\ell(w)} \ol{\Sch}_{s_kw}.
\end{equation}

These operators on $\ol{R}$ have been studied previously. $\nabla$ was defined (on polynomials) by Stanley \cite{Stanley-schubert-shenanigans}, while $\xi$ was defined by Nenashev \cite{Nenashev}. Similar operators were previously studied by Kerov; see also \cite{Okounkov-z-measures, HamakerPechenikSpeyerWeigandt}. Nenashev showed that both $\nabla$ and $\xi$ satisfy the Leibniz rule:
\begin{proposition}\cite[Propositions~4,~5]{Nenashev}
For any $f,g\in\ol{R}$,
\[
\nabla(fg) = \nabla(f)g + f\nabla(g) \qquad\qquad \text{and} \qquad\qquad \xi(fg) = \xi(f)g + f\xi(g).
\]
\end{proposition}

\begin{remark}
    In fact, $\xi$ can be defined on the larger space $\ol{Q}$, on which it still satisfies the Leibniz rule, although we will not need that here.
\end{remark}
We claim that the definition \eqref{eq:nabla-xi-olR-action} is canonical.
\begin{theorem} \label{thm:nabla-xi-preseve-equiv}
For any $\bsvec{a}, \bsvec{b} \in \QAlg$, $\bsvec{a} \equiv \bsvec{b}$ implies that $\xi \bsvec{a} \equiv \xi \bsvec{b}$, and if $\bsvec{a}, \bsvec{b} \in \mcR$, $\bsvec{a} \equiv \bsvec{b}$ implies $\nabla \bsvec{a} \equiv \nabla \bsvec{b}$.
\end{theorem}

\begin{proof}
    We first consider $\xi$. By linearity, it suffices to show that the statement is true for any $\bs{p}, \bs{q}\in \Zcol^*$. If $\bs{p} \equiv \bs{q}$, then $m(\bs{p}) = m(\bs{q})$. Deleting the first character from $\bs{p}$ and $\bs{q}$ simply deletes the first character from their maximal bottom rows, so
    \[
    m(\xi(\bs{p})) = \xi(m(\bs{p})) = \xi(m(\bs{q})) = m(\xi(\bs{q})),
    \]
    and $\xi(\bs{p}) \equiv \xi(\bs{q})$.

    Next, we consider $\nabla$. By \eqref{eq:word-poly-of-Schubert-vector}, \eqref{eq:nabla-xi-Q-action} and \eqref{eq:nabla-xi-olR-action},
    \begin{equation} \label{eq:nabla-same-as-in-words}
    \nabla \olF_{\Schvec(w)} = \nabla \ol{\Sch}_w = \sum_{k, s_kw<w} k\ol{\Sch}_{s_kw} = \olF_{\sum_{k, s_kw<w} k\Schvec(s_kw)} = \olF_{\nabla \Schvec(w)},
    \end{equation}
    and extending by linearity, the same is true for any element of $\mcS$, and so
    \begin{equation} \label{eq:nabla-apply-to-S}
    \nabla(\olF_{\bsvec{a}}) = \nabla(\olF_{\mcS(\bsvec{a})}) = \olF_{\nabla(\mcS(\bsvec{a}))}, \qquad\qquad \text{for all }\bsvec{a}\in\mcR.
    \end{equation}

    Consider the case where $\bsvec{a} = \Schvec(w_1)\Schvec(w_2)\cdots\Schvec(w_k)$ is a basis vector of $\mcR$. Using \eqref{eq:nabla-same-as-in-words} and \eqref{eq:nabla-apply-to-S}, along with Proposition \ref{prop:word-poly-multiplication} and the Leibniz rules for $\nabla$ on both $\mcR$ and $\ol{R}$, we have
    \begin{align*}
        \olF_{\nabla(\Schvec(w_1)\cdots\Schvec(w_k))}
        \\&\hspace{-30pt}= \olF_{\nabla(\Schvec(w_1))\Schvec(w_2)\cdots\Schvec(w_k) + \cdots + \Schvec(w_1)\cdots\Schvec(w_{k-1})\nabla(\Schvec(w_k))}
        \\&\hspace{-30pt}= \olF_{\nabla(\Schvec(w_1))}\olF_{\Schvec(w_2)}\cdots\olF_{\Schvec(w_k)} + \cdots + \olF_{\Schvec(w_1)}\cdots\olF_{\Schvec(w_{k-1})}\olF_{\nabla(\Schvec(w_k))}
        \\&\hspace{-30pt}= \nabla(\olF_{\Schvec(w_1)})\olF_{\Schvec(w_2)}\cdots\olF_{\Schvec(w_k)} + \cdots + \olF_{\Schvec(w_1)}\cdots \olF_{\Schvec(w_{k-1})}\nabla(\olF_{\Schvec(w_k)})
        \\&\hspace{-30pt}= \nabla(\olF_{\Schvec(w_1)}\cdots\olF_{\Schvec(w_k)})
        \\&\hspace{-30pt}= \nabla\olF_{\Schvec(w_1)\cdots\Schvec(w_k)}
        \\&\hspace{-30pt}= \olF_{\nabla(\mcS(\Schvec(w_1)\cdots\Schvec(w_k))}.
    \end{align*}
    so we have $\nabla \bsvec{a} \equiv \nabla \mcS(\bsvec{a})$. By linearity, this is true for any $\bsvec{a}\in\mcR$, and therefore if $\bsvec{a}\equiv\bsvec{b}$, then $\nabla \bsvec{a} \equiv \nabla\mcS(\bsvec{a}) = \nabla\mcS(\bsvec{b}) \equiv \nabla\bsvec{b}$.
\end{proof}

\begin{corollary} When $\nabla$ or $\xi$ is restricted to elements of $\mcR$, then $\phi \nabla = \nabla \phi$ and $\phi \xi = \xi \phi$.
\end{corollary}

\begin{proof}
Let $\bsvec{a}\in\mcR$ and let $f = \phi(\bsvec{a})$. Using Theorem~\ref{thm:nabla-xi-preseve-equiv} and \eqref{eq:nabla-apply-to-S},
    \[
    \phi (\nabla \bsvec{a}) = \olF_{\nabla\bsvec{a}} = \olF_{\nabla\mcS(\bsvec{a})} = \nabla\olF_{\mcS(\bsvec{a})} = \nabla f = \nabla \phi (\bsvec{a}). 
    \]
The same argument holds for $\xi$, since the analogue to \eqref{eq:nabla-apply-to-S}, $\xi(\olF_{\bsvec{a}}) = \xi(\olF_{\mcS(\bsvec{a})}) = \olF_{\xi(\mcS(\bsvec{a}))}$, holds for the same reasons.
\end{proof}

\begin{remark}
    Theorem~\ref{thm:nabla-xi-preseve-equiv} is much easier to show for $\xi$ than for $\nabla$, in part because the result only holds for $\nabla$ on the subalgebra $\mcR$ rather than the whole chromatic shuffle algebra.
    It would be interesting to know if there is a way to modify $\nabla$ such that $\bsvec{a} \equiv \bsvec{b} \implies \nabla \bsvec{a} \equiv \nabla \bsvec{b}$ holds on all of $\QAlg$.
\end{remark}

If $\bsvec{a}, \bsvec{b} \in \mcR$ are bottom-row-equivalent, then by linearity, $\zeta \bsvec{a} \equiv \zeta \bsvec{b}$ for any operator $\zeta$ formed by taking linear combinations and/or repeated applications of $\xi$ and $\nabla$.

\begin{example}
    If $\bsvec{a}, \bsvec{b} \in \mcR$ and $\bsvec{a}\equiv\bsvec{b}$, then by Theorem~\ref{thm:nabla-xi-preseve-equiv} the sum of the values of the $i$th entries of the words in $\bsvec{a}$ (with multiplicity) equals the corresponding sum for $\bsvec{b}$. For instance,
    $\Sch_{s_2}\Sch_{s_2} = \olF_{2^{[1]}2^{[2]} + 2^{[2]}2^{[1]}} = \Sch_{s_3s_2} + \Sch_{s_1s_2} = \olF_{32 + 12}$, so in terms of words,
    \[
    2^{[1]}2^{[2]} + 2^{[2]}2^{[1]} \equiv 32 + 12.
    \]
    Applying $\xi\nabla$ to both sides, we get $\xi\nabla(2^{[1]}2^{[2]} + 2^{[2]}2^{[1]}) = (2 + 2)\emptyset \equiv \xi\nabla(32 + 12) = (3 + 1)\emptyset$.
\end{example}

The next proposition uses this technique to give a variant of Proposition \ref{prop:nenashev}.

\begin{proposition} \;
    \begin{equation} \label{eq:sigmak-Schu-Schv}
    \binom{\ell(u) + \ell(v)}{\ell(v)} (\rho(u) |\RW(v)| + |\RW(u)| \rho(v)) = \sum_w \ol{c_{u,v}^w} \rho(w),
    \end{equation}
    where for $w \in S_\Z$,
    \[
    \rho(w) = \sum_{(w_1, \ldots, w_{\ell(w)}) \in \RW(w)} (w_1 + w_2 + \ldots + w_{\ell(w)}).
    \]
\end{proposition}
\begin{proof}
    Define
    \[
    \sigma_k = \nabla \xi^{k-1} + \xi \nabla \xi^{k-2} + \ldots + \xi^{k-1} \nabla.
    \]
    It is easy to check that for $\bs{p} \in \Zcol^k$, we have $\sigma_k \bs{p} = (\val(\bs{p})_1 + \val(\bs{p})_2 + \ldots +  \val(\bs{p})_k)\emptyset$. Letting $k = \ell(u) + \ell(v)$, $\sigma_k (\Schvec(u) \Schvec(v))$ is the sum of all the values in all shuffles of all reduced words of $u$ and $v$. Consider some reduced word $\bs{u} \in \RW(u)$. There are exactly $\binom{\ell(u) + \ell(v)}{\ell(v)}|\RW(v)|$ shuffles in the product $\Schvec(u)\Schvec(v)$ which contain $\bs{u}$ as a subword (first choose which entries of the shuffle are the $\bs{u}$-entries, and then choose which reduced word of $v$ occupies the other entries). Applying $\sigma_k$ and summing over all $\bs{u}\in\RW(u)$, the total contribution to $\sigma_k (\Schvec(u) \Schvec(v))$ of all the characters in all the reduced words for $u$ is $\binom{\ell(u) + \ell(v)}{\ell(v)}\rho(u)|\RW(v)|$. Similar logic shows that the characters in $v$ contribute $\binom{\ell(u) + \ell(v)}{\ell(v)}|\RW(u)|\rho(v)$ to the sum. Thus, by Theorem~\ref{thm:nabla-xi-preseve-equiv},
    \begin{align*}
    \binom{\ell(u) + \ell(v)}{\ell(v)} (\rho(u) |\RW(v)| + |\RW(u)| \rho(v)) \emptyset
    &= \sigma_k (\Schvec(u) \Schvec(v))
    \\&\equiv \sigma_k (\sum_w \ol{c_{u,v}^w} \Schvec(w))
    = \sum_w \ol{c_{u,v}^w} \rho(w) \emptyset.
    \end{align*}
    Because the left side and the right side are constant multiples of the empty word, the equivalence implies equality.
\end{proof}

\section{Proof of the back-stabilization conjecture} \label{sec:backstabilization-thm-proof}
For this section, fix $w \in S_{\Z}$ and $u,v \in S_{\Z_+}$. We let $\psi$ be an equivalence-preserving multiset bijection from $\Schvec(u)\Schvec(v)$ to $\sum_w \ol{c_{u,v}^w} \Schvec(w)$, as guaranteed by Proposition~\ref{prop:multiset-bijection}.

\subsection{DC-triviality}

Our arguments make use of an observation of Knutson known as DC-triviality. Recall the equivalent conditions for a descent specified by Lemma \ref{lem:desdefs}.

\begin{lemma}[DC-triviality,~\cite{Knutson-descent-cycling}] \label{lem:Knutson-dc-triv}
    If $w$ has a descent at $i$ but $u,v$ don't have descents at position $i$, then $\ol{c_{u,v}^w} = 0$.
\end{lemma}

\begin{proof}
    Suppose $w$ has a descent at $i$ and $\ol{c_{u,v}^w} \ne 0$. Let $k=\ell(w)$ and pick any $\bs{w} \in \RW(w)$ such that $\bs{w}_k = i$. Then $\bs{w}\in \mcS(\Schvec(u)\Schvec(v))$. Let $\bs{p} = \psi^{-1}(\bs{w}) 
    \in \Schvec(u)\Schvec(v)$. $m(\bs{w}) = m(\bs{p})$ so $\val(\bs{w}_k) = \val(\bs{p}_k) = i$, but $u$ and $v$ don't have descents at $i$ so no words in $\Schvec(u)\Schvec(v)$ end in a character with value $i$, leading to a contradiction.
\end{proof}

DC-triviality is important for back-stabilization since $u$ and $v$ have no descents at nonpositive indices. Therefore, if $\ol{c_{u,v}^w}\ne 0$, $w$ also cannot have a descent at a nonpositive index. This observation leads to the following two corollaries.

\begin{corollary} \label{cor:dc1}
    Given $w$ with $\ol{c_{u,v}^w} \neq 0$, if $\theta_i(w) = 1$, then for all $i \leq j \leq 1$, $\theta_j(w) = 1$.
\end{corollary}
\begin{proof}
    Let $j$ be the minimum index such that $j>i$ and $\theta_j(w) = 0$. Then, $\text{code}(w)_{j - 1} > 0 = \text{code}(w)_j$, so $w$ has a descent at $j-1$. By DC-triviality and the above observation, $w$ has no descents at nonpositive indices, so we must have $j>1$.
\end{proof}
\begin{corollary} \label{cor:dc2} 
     If $c_{u,v}^w \neq 0$ and $w \notin S_{\Z_+}$, then for any $\bs{w} \in \RW(w)$ and any $n \leq 0$, if $n \in \supp\bs{w}$ and $n-1 \notin \supp\bs{w}$, then $\BS(w) = 1-n$.
\end{corollary}

Note that whether a given $n$ appears or not in $\bs{w}$ is an invariant of $w$; $n$ either appears in all reduced words for $w$ or none of them.

\begin{proof} $n-1 \notin \supp\bs{w}$ implies $\theta_{n-1}(w) = 0$ since for $j > n-1$, we have $w(j) = s_{\bs{w}_1}s_{\bs{w}_2}s_{\bs{w}_\ell(w)}(j)$, which can inductively seen to be greater than $n-1$, and we also have inductively $s_{\bs{w}_1}s_{\bs{w}_2}s_{\bs{w}_\ell(w)}(n-1) \leq n-1$.

By Corollary \ref{cor:dc1}, for all $n' < n-1$, we also have $\theta_{n'}(w) = 0$. $n < 1$ so $w \notin S_{\Z_+}$ so we can apply \eqref{eq:BS-def} to get
    \[
    \BS(w) = 1 - \min(\{j \;|\; \theta_j(w) = 1\}) \geq 1-n.
    \]
    Conversely, $n \in \supp\bs{w}$, so $\BS(w) = 1 - \min(\supp \bs{w}) \leq 1-n$, and we have equality.
\end{proof}

\subsection{Increasing suffixes}

Given a word $\bs{p} = (p_1, p_2, \ldots, p_k)$ and indices $1\le i\le j\le k$, let $\bs{p}_{i,j}$ be the contiguous subword $(p_i, p_{i+1},\cdots p_{j-1}, p_j)$ of $\bs{p}$ starting with $p_i$ and ending with $p_j$.

An \emph{increasing suffix} of $\bs{p}$ is a suffix of $\bs{p}$ that is (strictly) increasing. Let $\IncSuf(\bs{p})$ be the maximal-length increasing suffix of $\bs{p}$, which we will refer to simply as the increasing suffix of $\bs{p}$. The \textit{increasing suffix length} $I(\bs{p})$ of $\bs{p}$ is the length of its increasing suffix,\[I(\bs{p}) = \ell(\IncSuf(\bs{p})) = \max(\{i \;|\; p_{k-i+1} < p_{k-i+2} < \ldots < p_k\}).\] The increasing suffix of $\bs{p}$ can also be written $\IncSuf(\bs{p}) = \bs{p}_{k-I(\bs{p})+1, k}$. 

Given an integer $i\in \Z$, the \emph{increasing $i$-suffix} of $\bs{p}$ is 
\[
\IncSuf_i(\bs{p}) =
\begin{cases} \IncSuf(\bs{p}), & \text{if } \val(\bs{p}_k) \le i, \\ \emptyset, & \text{if } \val(\bs{p}_k) > i.
\end{cases}
\]
Its length, the \textit{increasing $i$-suffix length} of $\bs{p}$ is given by
\[
I_i(\bs{p}) = \ell(\IncSuf_i(\bs{p})) = \begin{cases}  I(\bs{p}), & \text{if } \val(\bs{p}_k) \le i, \\ 0, & \text{if } \val(\bs{p}_k) > i.
\end{cases}
\]

Increasing suffixes inform us about the slide polynomial decomposition of Schubert products. In particular, increasing $i$-suffix lengths are invariants of slide polynomials.

\begin{proposition} \label{prop:increasing-suffix-invariant}
    Given words $\bs{p}, \bs{q} \in \Zcol^*$, if $\bs{p} \equiv \bs{q}$, then for all $i$, $I_i(\bs{p}) = I_i(\bs{q})$.
\end{proposition}
\begin{proof}
    $\bs{p} \equiv \bs{q}$ implies that $m(\bs{p}) = m(\bs{q})$, so we prove the result by showing that $\bs{p}$ has the same increasing $i$-suffix length as its maximal bottom row, that is, \[I_i(\bs{p}) = I_i(m(\bs{p})) \qquad\qquad \text{for all } i.\]

    This fact follows from Lemma \ref{lem:max-bottom-inc}, since $\bs{p}$ and $m(\bs{p})$ have increases at exactly the same entries, and moreover have the same final entry.
\end{proof}

For elements of $\Z^*$, we get the following stronger statement:
\begin{proposition} \label{prop:Z-star-increasing-suffix}
    For $\bs{p}, \bs{q} \in \Z^*$, if $\bs{p} \equiv \bs{q}$, then  $\IncSuf(\bs{p}) = \IncSuf(\bs{q})$.
\end{proposition}

\begin{proof}
    It suffices to show that $\IncSuf(\bs{p}) = \IncSuf(m(\bs{p}))$ because then $\IncSuf(\bs{p}) = \IncSuf(m(\bs{p})) = \IncSuf(m(\bs{q})) = \IncSuf(\bs{q})$.
    
    To prove this, we will show that for $0 \leq j < I(\bs{p})$, we have $m(\bs{p})_{k-j} = \bs{p}_{k-j}$. When $j=0$ we have by \eqref{eq:max-bot-row} that $m(\bs{p})_{k-0} = \val(\bs{p}_{k-0}) = \bs{p}_{k-0}$. Then for $0<j<I(\bs{w})$ we have $\bs{p}_{k-j} < \bs{p}_{k-j+1}$, and $\bs{p} \in \Z^*$, so this implies that $\bs{p}_{k-j} \leq \bs{p}_{k-j+1} - 1$. We can inductively assume $m(\bs{p}_{k-j+1}) = \bs{p}_{k-j+1}$, and so by \eqref{eq:max-bot-row}, we have $m(\bs{p})_{k-j} = \min(\val(\bs{p}_{k-j}), m(\bs{p})_{k-j+1} - 1) = \min(\bs{p}_{k-j}, \bs{p}_{k-j+1} - 1) = \bs{p}_{k-j}$.
\end{proof}

\begin{example}
    $3^{[1]}2^{[1]}2^{[2]} \equiv 312 \equiv 212$, and $\IncSuf(3^{[1]}2^{[1]}2^{[2]}) = 2^{[1]}2^{[2]}$, and $\IncSuf(312) = \IncSuf(212) = 12$. $I_i(3^{[1]}2^{[1]}2^{[2]}) = I_i(312) = I_i(212) = \begin{cases}
        2 & i \geq 2, \\
        0 & \text{otherwise}.
    \end{cases}$
\end{example}

Proposition~\ref{prop:Z-star-increasing-suffix} directly implies the following fact:
\begin{proposition} \label{prop:psi-maximal-inc-suf}
Let $\psi:\mathcal{M}\to\RW$ satisfy Proposition~\ref{prop:magic-map-exists}. For any $\bs{p} \in \mathcal{M}$, for all $i$,
\begin{equation}
    I_i(\bs{p}) = I_i(\psi(\bs{p})).
\end{equation}

Furthermore,
\begin{equation}
    \IncSuf(m(\bs{p})) = \IncSuf(\psi(\bs{p})).
\end{equation}
\end{proposition}

Given some set or multiset of words $X$, let the increasing $i$-suffix length of that set be $I_i(X) = \max_{\bs{p} \in X} I_i(\bs{p})$, the maximal $i$-suffix length of any word in $X$. Using our convention of  treating elements of $\N[\Zcol^*]$ as multisets, we see that increasing $i$-suffix length is an invariant of $\N$-linear sums of slide polynomials:

\begin{lemma} \label{lem:equiv-preserves-I}
    Given $\bsvec{a}, \bsvec{b} \in \N[\Zcol^*]$ with $\bsvec{a} \equiv \bsvec{b}$, we have $I_i(\bsvec{a}) = I_i(\bsvec{b})$ for all $i$.
\end{lemma}

\begin{proof}
    Since $\bsvec{a} \equiv \bsvec{b}$, there exists a maximal-bottom-row preserving bijection $f: \bsvec{a} \rightarrow \bsvec{b}$, so
    \[I_i(\bsvec{a}) = \max_{\bs{p} \in \bsvec{a}} I_i(\bs{p}) = \max_{\bs{p} \in \bsvec{a}} I_i(f(\bs{p})) = \max_{\bs{q} \in \bsvec{b}} I_i(\bs{q}) = I_i(\bsvec{b}),\]
    where the second equality uses Proposition \ref{prop:increasing-suffix-invariant}.
\end{proof}

Increasing $i$-suffixes also behave well with respect to multiplication in $\N[\Zcol^*]$. In particular,

\begin{proposition} \label{prop:multiplication-adds-I}
    For  $\bsvec{a}, \bsvec{b} \in \N[\Zcol^*]$, $I_i(\bsvec{a}\bsvec{b}) = I_i(\bsvec{a}) + I_i(\bsvec{b})$
\end{proposition}
\begin{proof}
        It suffices to show that for $\bs{p}, \bs{q} \in \Zcol^*$, $I_i(\bs{p}\bs{q}) = I_i(\bs{p}) + I_i(\bs{q})$; taking maximums over $\bs{p}\in\bsvec{a}$ and $\bs{q}\in\bsvec{b}$ gives the result.
        
        Write $\bs{p} = \bs{c} \circ \IncSuf_i(\bs{p})$, and  $\bs{q} = \bs{d} \circ \IncSuf_i(\bs{q})$. There is exactly one shuffled word $\bs{e}$ in the product $\IncSuf_i(\bs{p})\IncSuf_i(\bs{q})$ which is increasing. Then $\bs{c} \circ \bs{d} \circ \bs{e}$ is a shuffled word in $\bs{p}\bs{q}$ whose increasing $i$-suffix has length $\geq I_i(\bs{p}) + I_i(\bs{q})$, which establishes $I_i(\bs{p}\bs{q}) \geq I_i(\bs{p}) + I_i(\bs{q})$.

        To show the other inequality, consider any $\bs{r} \in \bs{p} \bs{q}$. The $\bs{p}$-entries of the increasing $i$-suffix of $\bs{r}$ are an increasing $i$-suffix of $\bs{p}$, and similarly the $\bs{q}$-entries of the increasing $i$-suffix of $\bs{r}$ are an increasing $i$-suffix of $\bs{q}$. So, $I_i(\bs{r}) \leq I_i(\bs{p}) + I_i(\bs{q})$.
\end{proof}

Combining Lemma \ref{lem:equiv-preserves-I} and Proposition \ref{prop:multiplication-adds-I}, we have
\begin{theorem} \label{thm:increasing-suffix-product}
    For all $i \in \Z$, we have 
    \[
    \max_{w \;|\; \ol{c_{u,v}^w} \ne 0}I_i(\RW(w)) = I_i(\RW(u)) + I_i(\RW(v)).
    \] In particular, if for some $i$, $I_i(\RW(w)) > I_i(\RW(u)) + I_i(\RW(v))$, then $\ol{c_{u,v}^w}=0$.
\end{theorem}

\begin{proof}
    We have
    \[
    \max_{w \;|\; \ol{c_{u,v}^w} \ne 0}I_i(\RW(w)) = I_i\left(\sum_w \ol{c_{u,v}^w} \Schvec(w)\right) = I_i\left(\Schvec(u)\Schvec(v)\right) = I_i(\RW(u)) + I_i(\RW(v)),
    \]
    where the first equality is by definition, and the third is by Proposition \ref{prop:multiplication-adds-I}. For the second equality, $\Schvec(u)\Schvec(v) \in \N[\Zcol^*]$ and using the positivity of the Schubert structure constants, $\sum_w \ol{c_{u,v}^w} \Schvec(w) \in \N[\Zcol^*]$, so we can use Lemma \ref{lem:equiv-preserves-I}.
\end{proof}

\subsection{Increasing suffixes of reduced words} \,

The previous result suggests the following definitions:
\[
I(w) = I(\Schvec(w)), \qquad\qquad I_i(w) = I_i(\Schvec(w)),
\]
and
\[
I(u,v) = I(\Schvec(u)\Schvec(v)), \qquad\qquad I_i(u,v) = I_i(\Schvec(u)\Schvec(v)).
\]
$I(w)$ is the length of the longest increasing suffix in any reduced word for $w$, $I(u,v)$ is the length of the longest increasing suffix in any reduced word of any permutation that appears in the Schubert product $\ol{\Sch}_u\ol{\Sch}_v$, and similarly for $I_i(w)$ and $I_i(u,v)$.

To prove Theorem~\ref{thm:backstabilization-thm}, we carefully analyze the increasing suffixes of words in $\Schvec(u)\Schvec(v)$ and in $\sum_w \ol{c_{u,v}^w} \Schvec(w)$. 

Increasing suffixes of reduced words have a close relationship with the Lehmer code. We have the following properties:

\begin{proposition} \label{prop:reducedIncreasingSequences}
    Given permutation $w$ with $\ell(w) = k$,
    \begin{enumerate}
        \item[(a)] For $\bs{w} \in \RW(w)$, if $j$ is in $\IncSuf(\bs{w})$ then $\theta_j(w) = 1$.
        \item[(b)] For $i \in \mathbb{Z}$, $I_i(w) \leq \lambda_i(w)$.
        \item[(c)] For $i \in \mathbb{Z}$, if $\theta_i(w) = 0$, then $I_{i-1}(w) = I_i(w) = \lambda_i(w)$.
    \end{enumerate}
\end{proposition}
\begin{proof} \;
    \begin{enumerate}
        \item[(a)] Induction on $\ell(w)$. The base case is when $\ell(w) = 1$, and the claim is clear.
        
        If $\bs{w} = \bs{w}_1, \ldots, \bs{w}_k$, then $w$ has a descent at $\bs{w}_k$, so by Lemma \ref{lem:desdefs}(b) $\text{code}(w)_{\bs{w}_k} > \text{code}(w)_{\bs{w}_k+1} \geq 0$. If $\bs{w}_{k-1} \geq \bs{w}_k$, then we are done since $\bs{w}_k$ is the entire increasing suffix. Assume we are in the other case where $\bs{w}_{k-1} < \bs{w}_k$. We can use the inductive hypothesis on $\bs{w}_1, \ldots, \bs{w}_{k-1}$, which is a reduced word for $w s_{\bs{w}_k}$, to show that for any $j$ in $\IncSuf(\bs{w}_1, \ldots, \bs{w}_{k-1})$,  $\text{code}(w s_{\bs{w}_k})_j > 0$. Note that in this case any $j$ in $\IncSuf(\bs{w}_1, \ldots, \bs{w}_{k-1})$ is less than $\bs{w}_{k}$, so $\text{code}(w)_j = \text{code}(w s_{\bs{w}_k})_j > 0$ by Lemma \ref{lem:desdefs}(c). Any $j$ in $\IncSuf(\bs{w}_1, \ldots, \bs{w}_{k})$ is either $\bs{w}_k$ or is in $\IncSuf(\bs{w}_1, \ldots, \bs{w}_{k-1})$, so we have proved the claim.

        \item[(b)] For any $\bs{w} \in \RW(w)$, if $I_i(\bs{w}) = 0$, then $I_i(\bs{w}) \leq \lambda_i(w)$ since the latter is nonnegative. If $I_i(\bs{w})>0$, then $\bs{w}_k \leq i$, so for each $j$ in $\IncSuf(\bs{w})$, $j \leq i$ and $\theta_j(w) = 1$ by part (a). Thus, $I_i(\bs{w}) = \ell(\IncSuf(\bs{w})) \leq \sum_{k={-\infty}}^i \theta_k = \lambda_i$.
    
        \item[(c)] If $\theta_i(w) = \text{code}(w)_i = 0$, then $w$ does not have a descent at $i$, so no reduced words for $w$ end in $i$, so $I_{i-1}(w) = I_i(w)$. Now we just have to show that $I_i(w) = \lambda_i(w)$. From part (b), we already have $I_i(w) \leq \lambda_i(w)$, so we just have so establish $I_i(w) \geq \lambda_i(w)$. To do this we induct over $\lambda := \lambda_i(w)$.
        
        The base case is if $\lambda = 0$, in which case the statement is clear. Otherwise let $j = \max(\{\alpha \leq i \;|\; \text{code}(w)_{\alpha} > 0\})$. By the maximality of $j$, $\text{code}(w)_{j} > 0 = \text{code}(w)_{j+1}$ (in particular, $\lambda = \lambda_j(w)$). Thus, $w$ has a descent at $j$, so by Lemma \ref{lem:desdefs}(c), $\text{code}(ws_j)_j = \text{code}(w)_{j+1} = 0$, and for all $j' < j$, $\text{code}(ws_j)_{j'} = \text{code}(w)_{j'}$. Thus, we can apply the inductive hypothesis and conclude that
        \[
        I_j(ws_j) = \lambda_j(ws_j) = \lambda_j(w) - 1 = \lambda - 1.
        \]
        Let $\bs{p} \in \RW(ws_j)$ be a word that achieves $I_j(\bs{p}) = I_j(ws_j) = \lambda-1$. Then since $j\le i$, the word $\bs{p}_1, \ldots, \bs{p}_{k-1}, j$ is a reduced word for $w$ which has an increasing $i$-suffix length equal to $I_j(ws_j) + 1 = \lambda$, so $I_i(w) \geq  \lambda$ as desired.\qedhere
    \end{enumerate}
\end{proof}

Taking $i$ to be large, we obtain the following corollary of Theorem \ref{thm:increasing-suffix-product} and Proposition \ref{prop:reducedIncreasingSequences}:

\begin{corollary} \label{cor:lehmer-nonzero-relationship}
    If $\ol{c_{u,v}^w} \neq 0$, then
    \[
    \lambda_\infty(w) \leq \lambda_\infty(u) + \lambda_\infty(v)
    \]
    i.e. the number of nonzero rows in the Lehmer code of $w$ is less than or equal to the number of nonzero rows in $u$ plus the number of nonzero rows in the Lehmer code of $v$.
\end{corollary}

This gives a simple criteria to determine the vanishing of some Schubert structure constants. See \cite{dizier2022generalized} for an overview of existing results on the topic.

Different reduced words of $w$ may have different increasing $i$-suffixes, but it turns out that maximal increasing $i$-suffixes are unique.

\begin{corollary}
\label{cor:max-inc-suffix}
Let $w\in S_\Z$.

\begin{enumerate}
    \item[(a)] There is a unique maximal-length increasing suffix for $w$, which is given by $s_{i_1}\cdots s_{i_m}$, where $i_1<\ldots<i_m$ are the values $i$ where $\theta_i(w) = 1$.
    \item[(b)] Every increasing suffix for $w$ is a subword of the maximal increasing suffix.
\end{enumerate}
\end{corollary} 

\begin{proof}
    This follows from Proposition \ref{prop:reducedIncreasingSequences}, parts (a) and (c), taking $i$ to be large in (c).
\end{proof}

We also make the following observation. Let $w\in S_\Z$ and let $\bs{p}$ be the maximal increasing suffix of $w$. $\bs{p}$ has the form
\[
\bs{p} = (i_1, i_1+1, \ldots, i_1+k_1, i_2, i_2+1,\ldots, i_2+k_2, \ldots, i_t, i_t+1, \ldots, i_t+k_t),
\]
where $i_j \ge i_{j-1} + k_{j-1} + 2$. Write $\bs{q}^{(j)} = (i_j, i_j+1,\ldots, i_j+k_j)$.
Then since $s_is_j = s_js_i$ when $|i-j|\ge 2$, every word $\bs{q}^{(\sigma(1))}\circ \bs{q}^{(\sigma(2))}\circ\cdots\circ \bs{q}^{(\sigma(t))}$ consisting of the $\bs{q}^{(j)}$ in any order is a suffix of $w$.
Furthermore, every subword of $\bs{p}$ consisting of some of the $\bs{q}^{(j)}$ is an increasing suffix for $w$.

\begin{example}
Consider the permutation $w\in S_{\Z_+}$ with one-row notation $2136574$. Its Lehmer code is $(1,0,0,2,1,1,0)$. Then the maximal increasing suffix for $w$ is $\bs{p} = s_1s_4s_5s_6$, and every increasing suffix of $w$ is a subword of $\bs{p}$. In particular, $\bs{q}^{(1)} = s_1$ and $\bs{q}^{(2)} = s_4s_5s_6$ commute, and so $s_1$ is an increasing suffix of $w$.
\end{example}

Not every subword of the maximal increasing suffix for $w$ is itself an increasing suffix for $w$. A careful analysis of the Lehmer code might lead to a characterization, but this is outside the scope of the present paper.

\begin{problem}
    Characterize all increasing suffixes of a given permutation $w\in S_\Z$.
\end{problem}

The question of which reduced words have maximal increasing suffixes is also interesting. This word is usually not unique, but one might ask whether any combinatorially interesting words $\bs{w}$ have $I(\bs{w}) = I(w)$. One such word is given recursively by the following definition. Let $\IncSuf(w)$ be the maximal-length increasing suffix for $w$, and also treat $\IncSuf(w)$ as the permutation it is a reduced word for. Then the word $\text{MaxIncWord}(w)$ defined by
\[
\text{MaxIncWord}(w) = \text{MaxIncWord}(w(\IncSuf(w))^{-1})\circ\IncSuf(w)
\]
satisfies $I(\bs{w}) = I(w)$, and a similar property holds for any any prefix $\bs{w'}$ of $\bs{w}$ obtained by repeatedly stripping off maximal increasing suffixes. We expect that this word has been previously studied, but are unaware of a source.

\subsection{Proof of the Back-stabilization Theorem}

\begin{proposition} \label{prop:incmerge}
    Let $u,v \in S_{\Z_+}$ with $\ell(u) + \ell(v) = k$.
    \begin{enumerate}
        \item[(a)] $\max_{i \geq 0}(I_i(u,v) - i) = \max_{i \geq 0}(\lambda_i(u) + \lambda_i(v) - i)$
        \item[(b)] Let $i \geq 0$ be such $I_i(u,v) - i$ attains its maximum. Let $\bs{p} \in \Schvec(u)\Schvec(v)$ with $I_i(\bs{p}) = I_i(u,v)$. Then, either $\IncSuf_i(m(\bs{p})) = \emptyset$ or $\IncSuf_i(m(\bs{p})) = (i-I(\bs{p})+1, i-I(\bs{p})+2, \ldots, i)$.
    \end{enumerate}
\end{proposition}

Note that $\max_{i \geq 0}(\lambda_i(u) + \lambda_i(v) - i)$ is the right side of \eqref{eq:bs-formula}.

\begin{proof} \;
    \begin{enumerate}
        \item[(a)]
        By Proposition \ref{prop:multiplication-adds-I} and Proposition \ref{prop:reducedIncreasingSequences}(b),
        \[
        I_i(u,v) - i = I_i(u) + I_i(v) - i \leq \lambda_i(u) + \lambda_i(v) - i,
        \]
        Taking the maximum over all $i \geq 0$ we get
        \[
        \max_{i \geq 0}(I_i(u,v)-i) \leq \max_{i \geq 0}(\lambda_i(u) + \lambda_i(v) - i).
        \]
        Now, let $i \geq 0$ be such that it maximizes $I_i(u,v) - i$, and let
        \[
        j = \min\{\alpha \geq i \;|\; \theta_{\alpha + 1}(u) = \theta_{\alpha + 1}(v) = 0\}.
        \]
        By construction, for all $i < j' \leq j$ we have $\theta_{j'}(u) + \theta_{j'}(v) \geq 1$. Applying Proposition \ref{prop:reducedIncreasingSequences}(c),
        \begin{align*}
        I_{j}(u,v) - j &= \lambda_j(u) + \lambda_j(v) - j \\&= \lambda_i(u) + \lambda_i(v) + \left(\sum_{j'=i+1}^{j} \theta_{j'}(u) + \theta_{j'}(v)\right) - j \\&\geq \lambda_i(u) + \lambda_i(v) + (j-i) - j \\&= \lambda_i(u) + \lambda_i(v) - i.
        \end{align*}
        So, $\max_{i \geq 0}(I_i(u,v) - i) \geq \max_{i \geq 0}(\lambda_i(u) + \lambda_i(v) - i)$.

        \item[(b)] We are trying to show that $m(\bs{p})_{k-j} = i - j$ for all $0 \leq j < I_i(\bs{p})$. If $I_i(\bs{p}) = 0$ then the statement is vacuous, so we can assume this is not the case. Otherwise, we induct on $j$. For $j=0$, we need to show that $\val(\bs{p}_k) = i$. We assumed $I_i(\bs{p}) \neq 0$, so $\val(\bs{p}_k) \leq i$, and if $\val(\bs{p}_k) < i$, then \[I_{i-1}(u,v) - (i-1) = I_{i-1}(\bs{p}) - (i-1) = I_i(\bs{p}) - (i-1) = I_i(u,v) - i + 1,\] but this contradicts the choice of $i$. Thus, $m(\bs{p})_k = \val(\bs{p}_k) = i$.
        
        Now we must show for $1 \leq j < I(\bs{p})$, we have $m(\bs{p})_{k-j} = i - j$, and we inductively assume $i - j = m(\bs{p})_{k - (j-1)} - 1$. Since we are in the increasing suffix, by Lemma  \ref{lem:max-bottom-inc} we have $m(\bs{p})_{k-j} < m(\bs{p})_{k-(j-1)}$. In particular, $m(\bs{p})_{k-j} \leq i - j$. To prove equality, we will suppose $m(\bs{p})_{k-j} < i - j$ and show a contraction. By \eqref{eq:max-bot-row} we have $\val(\bs{p}_{k-j}) = m(\bs{p})_{k-j} < m(\bs{p})_{k-(j-1)} - 1 \leq \val(\bs{p}_{k-(j-1)}) - 1$.
        
        We break the increasing suffix into two pieces, $\bs{c} := \bs{p}_{k-I(p)+1, k-j}$ and $\bs{d} := \bs{p}_{k-(j-1), k}$. Since $\bs{c}$ and $\bs{d}$ are both increasing and since the first entry of $\bs{d}$ has value at least two more than the last entry of $\bs{c}$, every entry of $\bs{d}$ has value greater than that of every entry of $\bs{c}$ by at least 2. As a result, $\bs{d}\circ\bs{c}$ is also a suffix of some word $\bs{q}$ in $\Schvec(u)\Schvec(v)$, by doing the appropriate Coxeter commutations for words of $u$ and $v$.
        
        Consider the last entry, $\bs{q}_k$, of $\bs{q}$. This is also the last entry of $\bs{c}$, so it equals $\bs{p}_{k-j}$. Thus, $\val(\bs{q}_k) = \val(\bs{p}_{k-j}) = m(\bs{p})_{k-j} < i-j$. Since $\bs{c}$ is increasing, this means that $I_{i-j-1}(\bs{q}) \ge \ell(\bs{c}) = I_i(\bs{p})-j$, so \[I_{i-j-1}(\bs{q}) - (i-j-1) \ge I_i(\bs{p})-i+1 > I_i(\bs{p}) - i,\] contradicting the choice of $i$.\qedhere
    \end{enumerate}
\end{proof}

\begin{example}
Using one-line notation, let $u = 21543$, and let $v = 12453$. $\code(u) = (1,0,2,1)$ and $\code(v) = (0,0,1,1)$. The reduced words for $u$ are $3431$, $3413$, $3143$, $1343$, $4341$, $4314$, $4134$, and $1434$. The only reduced word for $v$ is $34$.

$\Schvec(u)\Schvec(v)$ consists of all of the colored shuffles between the reduced words for $u$ and the reduced words for $v$. By considering these shuffles, one can determine that $I_i(u,v) - i$ attains its maximum when $i=4$ and that $I_4(u,v) = 5$, due to the word \[
\bs{p} := 4^{[1]}1^{[1]}3^{[1]}3^{[2]}4^{[1]}4^{[2]} \in \Schvec(u)\Schvec(v),
\]
which has an increasing $4$-suffix of length $5$. $\lambda_i(u) + \lambda_i - i$ also attains its maximum when $i=4$, and $\max_{i \geq 0}(I_i(u,v) - i) = \max_{i \geq 0}(\lambda_i(u) + \lambda_i(v) - i) = 1$. We also have $m(\bs{p}) = (0,0,1,2,3,4)$, which satisfies $\IncSuf_4(m(\bs{p})) = (0, 1, 2, 3, 4)$.
\end{example}

We now prove \eqref{eq:bs-formula} by proving inequalities in both directions.

\begin{lemma}
    $\max_{i \geq 0} (I_i(u,v) - i) \leq \BS(u,v)$
\end{lemma}
\begin{proof} Let $i \geq 0$ be such that it maximizes $I_i(u,v) - i$. We have  $I_i(u,v) = I_i(\sum_w \ol{c_{u,v}^w}\Schvec(w))$, so let $\bs{w} \in \sum_w \ol{c_{u,v}^w}\Schvec(w)$ be such that $I_i(\bs{w}) = I_i(u,v)$. By Proposition \ref{prop:incmerge}, either $\IncSuf_i(\bs{w}) = \emptyset$, in which case $I_i(u,v) - i \leq 0 \leq \BS(u,v)$, so we are done, or $\IncSuf_i(\bs{w}) = (i - I_i(\bs{w}) + 1, i - I_i(\bs{w}) + 2, \ldots, i)$, so $i - I_i(\bs{w}) + 1 \in \supp \bs{w}$. So,
    \[
        \BS(w) = 1 - \min(\supp \bs{w}) \geq 1 - (i - I_i(\bs{w}) + 1) = \max_{i \geq 0} (I_i(u,v) - i).\qedhere
    \]
\end{proof}

\begin{lemma}
    $\max_{i \geq 0} (I_i(u,v) - i) \geq \BS(u,v)$
\end{lemma}
\begin{proof}
    If $\BS(u,v) = 0$ then the statement is clear. If $\BS(u,v) \neq 0$, then take any $w$ for which $\ol{c_{u,v}^w} \ne 0$ and $\BS(w) = \BS(u,v)$. By \eqref{eq:BS-def}, $\theta_{1 - \BS(u,v)}(w) = 1$. Let $j$ be the smallest integer such that $\theta_{j+1}(w) = 0$ and $j \geq 1 - \BS(u,v)$. Then, by Proposition \ref{prop:reducedIncreasingSequences}(c)
    \[
    I_j(w) = \lambda_j(w) = j - (1 - \BS(w)) + 1 = j + \BS(u,v),
    \]
    where the second equality is because $\theta_i = 0$ for $i < 1 - \BS(u,v)$.
    
    By Corollary \ref{cor:dc1}, $j\geq 1$. Thus we have
    \begin{align*}
    \max_{i \geq 0} (I_i(u,v) - i)
    &= \max_{i \geq 0} \left(I_i \left(\sum_w \ol{c_{u,v}^w}\Schvec(w) \right) - i\right)
    \\&\geq I_j(w) - j 
    \\&= j + \BS(u,v) - j
    \\&= \BS(u,v).\qedhere
    \end{align*}
\end{proof}

Combining these lemmas with Proposition \ref{prop:incmerge}(a), we have \[\BS(u,v) = \max_{i \geq 0} (I_i(u,v) - i) = \max_{i \geq 0}(\lambda_i(u) + \lambda_i(v) - i),\] and \eqref{eq:bs-formula} is proven.

Now for part (a) of Conjecture~\ref{conj:li-back-stable}:
\begin{lemma} \label{lem:crucial-last-step}
    Given $0 \leq m \leq \BS(u,v)$, there exists some $w\in S_\Z$ such that $\ol{c_{u,v}^w} \ne 0$ and $\BS(w) = m$.
\end{lemma}

The proof of this lemma is fairly technical, and combines every tool we have used so far. The main step is to apply the operator $(\nabla + m\xi)^{k - i - m}$ to both sides of \eqref{eq:Schub-word-equivalence}. On the left side, this gives a positive linear combination of shuffled words, and using the results of Sections \ref{sec:colored-shuffle-algebra} and \ref{sec:backstabilization-thm-proof} we show that there must exist the desired $w\in S_\Z$.

A proof of Conjecture \ref{conj:down-up MonkConnected} would result in a much less technical proof of this result.

\begin{proof}
    The fact that there exists a $w$ such that $\ol{c_{u,v}^w} \ne 0$ and $\BS(w) = 0$ is clear because non-back-stable Schubert expansions always have at least one term. So, we can assume that $\BS(u,v) > 0$ and we just have to prove the claim for $1 \leq m \leq \BS(u,v)$. To do this, by Corollary \ref{cor:dc2} we just have to show that for any $1 - \BS(u,v) \leq n \leq 0$ there exists some $w$ with $\ol{c_{u,v}^w} \ne 0$ such that for any reduced word $\bs{w} \in \RW(w)$, $n \in \supp\bs{w}$ but $n-1 \notin \supp\bs{w}$.
    
    Let $i \geq 0$ be such that it maximizes $I_i(u,v)-i$ and let $k = \ell(u)+\ell(v)$. We apply the operator $(\nabla - (n-1)\xi)^{k - (i - n) - 1}$ to both sides of \eqref{eq:Schub-word-equivalence}, and by Theorem \ref{thm:nabla-xi-preseve-equiv}, this yields
    \begin{equation} \label{eq:linear-comb-diff-on-Schub-prod}
    (\nabla - (n-1)\xi)^{k - (i - n) - 1}(\Schvec(u) \Schvec(v))
    \equiv (\nabla - (n-1)\xi)^{k - (i - n) - 1}\left(\sum_w \ol{c_{u,v}^w} \Schvec(w)\right).
    \end{equation}

    Both sides of \eqref{eq:linear-comb-diff-on-Schub-prod} can be computed directly via the definitions in \eqref{eq:nabla-xi-word-def}, and we obtain:
    \begin{equation} \label{eq:essential-eqn-main-proof}
    \sum_{\bs{p} \in \Schvec(u)\Schvec(v)}\mathcal{\pi}_{n-1}(\bs{p}_{1,k-(i-n)-1})\bs{p}_{k - (i - n), k}
    \equiv \sum_{w} \ol{c_{u,v}^w} \sum_{\bs{w} \in \RW(w)} \mathcal{\pi}_{n-1}(\bs{w}_{1,k-(i-n)-1})\bs{w}_{k - (i - n), k},
    \end{equation}
    where $\pi_c(\bs{q}) = (\val(\bs{q}_1) - c) (\val(\bs{q}_2) - c) \ldots (\val(\bs{q}_{\ell(\bs{q})}) - c)\in\Z$.
    
    This equation is the essential step of our proof. Note that on the left side of \eqref{eq:essential-eqn-main-proof},
    \[
    \mathcal{\pi}_{n-1}(\bs{p}_{1,k-(i-n)-1}) > 0 \qquad \text{ for all } \bs{p} \in \Schvec(u)\Schvec(v)
    \]
    because $\bs{p}_j \geq 1$ and $n < 1$, so $\bs{p}_j - (n-1) \geq 2$ for all $j$. So, the left side of \eqref{eq:essential-eqn-main-proof} is in $\N[\Zcol^*]$.
    
    By \eqref{eq:bs-formula}, $I_i(u,v)-i = \BS(u,v) > 0$, so $I_i(u,v) \neq 0$, and we can find $\bs{p} \in \Schvec(u)\Schvec(v)$ with $I_i(\bs{p}) = I_i(u,v) \neq 0$. $\bs{p}$ satisfies the hypotheses to apply Proposition \ref{prop:incmerge}(b), and $I_i(\bs{p}) \neq 0$ so
    \begin{align*}
    m(\bs{p})_{k - (I(\bs{p}) - 1), k} &= \IncSuf(m(\bs{p}))  \\
    &= (i - (I(\bs{p}) - 1), i - (I(\bs{p}) - 1) + 1, \ldots, i).
    \end{align*}

    In particular, $i - n \leq i - (1 - \BS(u,v)) = I_i(\bs{p}) - 1$, so
    \begin{align*}
        m(\bs{p}_{k - (i - n), k}) &= (i - (i-n), i - (i-n) + 1, \ldots, i) \\
        &= (n, n + 1, \ldots, i).
    \end{align*}
    Applying the operator $T_{(n, n + 1, \ldots, i)}$ from \eqref{eq:projection-op-def} to the left side of \eqref{eq:essential-eqn-main-proof} gives a positive integer, so the same must be true for the right side. This means that there must exist some $w$ with $\ol{c_{u,v}^w}\ne 0$ and some $\bs{w} \in \RW(w)$ with $\mathcal{\pi}_{n-1}(\bs{w}_{1,k-(i-n)-1}) \neq 0$ and $m(\bs{w}_{k - (i - n),k}) = (n, n + 1, \ldots, i)$. By Proposition \ref{prop:Z-star-increasing-suffix}, $\IncSuf(\bs{w}) = \IncSuf(m(\bs{w}))$, so $\bs{w}_{k - (i - n),k} = m(\bs{w}_{k - (i - n),k}) = (n, n + 1, \ldots, i)$, so $n \in \supp\bs{w}$ and $n-1 \notin \supp\bs{w}_{k - (i - n),k}$. Then $\mathcal{\pi}_{n-1}(\bs{w}_{1,k-(i-n)-1}) \neq 0$, so $n-1 \notin \supp\bs{w}_{1, k - (i - n) - 1}$, as desired.
\end{proof}

Recalling the definition of $V_k(u,v)$ from the introduction, the previous lemma shows that for all $0 \leq k \leq \BS(u,v)$, $V_k(u,v) \setminus V_{k+1}(u,v)$ is nonempty, so $\St(u,v) = \BS(u,v)$. This proves part (a) of Conjecture~\ref{conj:li-back-stable}, and therefore Theorem~\ref{thm:backstabilization-thm}.

\section{Forward Stabilization} \label{sec:forward-stabiization}

\subsection{Conjugation by $w_0$}

In this section, we prove our second main result, Theorem~\ref{thm:minimal-flag-variety}. We start by considering a duality for Schubert structure constants.

Let $\text{flip} : \Z \rightarrow \Z$ be the involution $i \mapsto 1-i$. Denote $\iota:S_\Z\to S_\Z$ to be conjugation by the flip map, $\iota(w) = \text{flip} \cdot w \cdot \text{flip}$ where multiplication is function composition. $\iota$ is an involution on $S_\Z$, and satisfies $\iota(uv) = \iota(u)\iota(v)$. For all $i$, $\iota(s_i) = s_{-i}$, so if $\bs{w} = s_{i_1} \ldots s_{i_k}$ then $\iota(\bs{w}) = s_{-i_1} \ldots s_{-i_k}$, and if $i_1 \ldots i_k \in \RW(w)$ then $-i_1 \ldots -i_k \in \RW(\iota(w))$.

Let $w_0:=w_0^{(n)}$ denote the longest element of $S_n$, and we suppress the superscript when $n$ is clear. We have
\begin{equation} \label{eq:iota-w0-conjugation}
w_0ww_0 = \gamma^n\iota(w) = \iota\gamma^{-n}(w)\in S_n, \qquad\qquad \text{for all } w\in S_n.
\end{equation}

The following geometric fact
is well-known to experts (see e.g. \cite{KnutsonYong}).

\begin{lemma} \label{lem:w0-cuvw-relation}
For all $u,v,w\in S_n$, $c_{u,v}^w = c_{w_0uw_0,w_0vw_0}^{w_0ww_0}$.
\end{lemma}

\begin{corollary} \label{cor:iota-cuvw-relation}
For all $u,v,w\in S_\Z$, $\ol{c_{u,v}^w} = \ol{c_{\iota(u),\iota(v)}^{\iota(w)}}$.
\end{corollary}

\begin{proof}
Fix $k$ and $n$ large enough such that $\gamma^k(u),\gamma^k(v),\gamma^k(w)\in S_n$. Then,
\[ \ol{c_{u,v}^w} = c_{\gamma^k(u),\gamma^k(v)}^{\gamma^k(w)}
= c_{w_0\gamma^k(u)w_0,w_0\gamma^k(v)w_0}^{w_0\gamma^k(w)w_0}
= c_{\iota\gamma^{k-n}(u),\iota\gamma^{k-n}(v)}^{\iota\gamma^{k-n}(w)}
= c_{\gamma^{n-k}\iota(u),\gamma^{n-k}\iota(v)}^{\gamma^{n-k}\iota(w)}
= \ol{c_{\iota(u),\iota(v)}^{\iota(w)}},
\]
where we have used \eqref{eq:iota-w0-conjugation}, Lemma~\ref{lem:w0-cuvw-relation}, and the definition of the back-stable structure constants.\qedhere
\end{proof}

It will be productive to work with a variant of the back-stabilization number, where negative numbers are allowed. For any $w\in S_\Z$, let $\til{\BS}(w)$ be the least integer $k$ such that $\gamma^k(w)\in S_{\Z_+}$, and for any $u,v\in S_\Z$, let
\begin{equation} \label{eq:real-back-stable-def}
\til{\BS}(u,v) = \max_{w|\ol{c_{u,v}^w}\ne 0} \til{\BS}(w) = \min\left\{k \;\left|\; \gamma^k(w)\in S_{\Z_+} \text{ for all } w\in S_{\Z} \text{ with } \ol{c_{u,v}^w}\ne 0\right.\right\}.
\end{equation}
For any $u,v\in S_{\Z_+}, w\in S_\Z$,
\[
\BS(w) = \max\left(\til{\BS}(w),0\right), \qquad\qquad \BS(u,v) = \max\left(\til{\BS}(u,v),0\right),
\]
and for any reduced word $\bs{w}\in\RW(w)$, $\til{\BS}(w) = 1-\min(\supp \bs{w})$.

Similarly let $\til{\FS}(w)$ be the least integer $k$ such that $\gamma^{-k}(w)\in S_{-\N}$, where $-\N$ is the set of nonpositive integers. For any reduced word $\bs{w}\in\RW(w)$, $\til{\FS}(w) = 1+\max(\supp \bs{w})$, and we have $\til{\FS}(w) = \til{\BS}(\iota w)$. If $w\in S_{\Z_+}$, $\til{\FS}(w) = \FS(w)$.

For any $u,v\in S_\Z$, let
\begin{equation} \label{eq:forward-stable-def}
\til{\FS}(u,v) = \max_{w|\ol{c_{u,v}^w}\ne 0} \til{\FS}(w) = \min\left\{k \;\left|\; \gamma^{-k}(w)\in S_{-\N} \text{ for all } w\in S_{\Z} \text{ with } \ol{c_{u,v}^w}\ne 0\right.\right\}.
\end{equation}

\begin{proposition} \label{prop:FS-BS-relationship}
Let $u,v\in S_\Z$, and let $a = \max(\til{\BS}(u),\til{\BS}(v))$. Then,
    \begin{enumerate}
        \item[(a)] $\til{\FS}(u,v) = \til{\BS}(\iota u,\iota v)$
        \item[(b)] $\til{\BS}(u,v) = \max_{i\ge -a}(\lambda_i(u)+\lambda_i(v)-i)$
        \item[(c)] $\max_{i \geq -a}(I_i(u,v) - i) = \max_{i \geq -a}(\lambda_i(u) + \lambda_i(v) - i)$
        \item[(d)] Let $i \geq -a$ be such $I_i(u,v) - i$ attains its maximum. Let $\bs{p} \in \Schvec(u)\Schvec(v)$ with $I_i(\bs{p}) = I_i(u,v)$. Then, either $\IncSuf_i(m(\bs{p})) = \emptyset$ or $\IncSuf_i(m(\bs{p})) = (i-I(\bs{p})+1, i-I(\bs{p})+2, \ldots, i)$.
    \end{enumerate}
\end{proposition}

\begin{proof} \;
    \begin{enumerate}
    \item[(a)] This follows from Corollary~\ref{cor:iota-cuvw-relation} and the fact that $\til{\FS}(w) = \til{\BS}(\iota w)$.
    \item[(b)] Up to a shift, this is just \eqref{eq:bs-formula}. We have $\gamma^a(u),\gamma^a(v)\in S_{\Z_+}$, and
    \begin{align*}
    \til{\BS}(u,v)
    &= a + \BS(\gamma^a(u),\gamma^a(v))
    \\&= a + \max_{i\ge 0}(\lambda_i(\gamma^a(u))+\lambda_i(\gamma^a(v))-i)
    \\&= \max_{i\ge -a}(\lambda_i(u)+\lambda_i(v)-i).
    \end{align*}
    (c) and (d) follow from applying $\gamma^{-a}$ to Proposition~\ref{prop:incmerge}\qedhere
    \end{enumerate}
\end{proof}

This brings us to the first main result of the section, a formula for $\til{\FS}(u,v)$. We start by defining the dual Lehmer code of a permutation.

\begin{definition} \label{def:dual-lehmer-code}
The \emph{dual Lehmer code} is the doubly-infinite sequence
\[
\text{dualcode}(w) = (\ldots, d_{-1}, d_0, d_1, \ldots),
\]
where
\[
d_i := \text{dualcode}(w)_i = |\{j\in \Z \;|\; j<i, w(j)>w(i)\}|,
\]
and let
\[\Theta_i(w) = \begin{cases}
    1 & \text{if dualcode}(w)_i>0, \\
    0 & \text{otherwise},
\end{cases}
\qquad\qquad
\Lambda_i(w) = \sum_{j\ge i} \Theta_j(w).
\]

\end{definition}
We have
\begin{equation} \label{eq:dualcode-relations}
\dualcode(w)_i = \code(w^{-1})_{w(i)} = \code(\iota w)_{1-i}, \qquad\qquad \Lambda_i(w) = \lambda_{1-i}(\iota w).
\end{equation}

\begin{example}
    Let $w = 2431$; then $w^{-1} = 4132$, and
    \[
    \code(w) = (1,2,1,0), \qquad \code(w^{-1}) = (3,0,1,0),
    \]
    \[
    \text{dualcode}(w) = (0,0,1,3), \qquad \text{dualcode}(w^{-1}) = (0,1,1,2).
    \]
    Applying $\iota$ (and shifting for convenience), $\gamma^4(\iota w) = 4213$ and $\gamma^4(\iota w^{-1}) = 3241$, and
    \[
    \code(\gamma^4(\iota w)) = (3,1,0,0), \qquad \code(\gamma^4(\iota w^{-1})) = (2,1,1,0),
    \]
    \[
    \text{dualcode}(\gamma^4(\iota w)) = (0,1,2,1), \qquad \text{dualcode}(\gamma^4(\iota w^{-1})) = (0,1,0,3).
    \]
\end{example}

\begin{corollary} \label{cor:forward-stabilization}
    \begin{equation} \label{eq:forward-stabilization}
    \til{\FS}(u,v) = \max_{i\le 1+\max(\til{\FS}(u),\til{\FS}(v))}(\Lambda_i(u)+\Lambda_i(v)+i-1).
    \end{equation}
\end{corollary}
\begin{proof}
    Combining \eqref{eq:dualcode-relations} and both parts of Proposition~\ref{prop:FS-BS-relationship},  we have
    \begin{align*}
    \til{\FS}(u,v)
    &=\max_{i\ge -\max(\til{\BS}(\iota u),\til{\BS}(\iota v))}(\lambda_i(\iota u)+\lambda_i(\iota v)-i)
    \\&=\max_{i\ge -\max(\til{\FS}(u),\til{\FS}(v))}(\Lambda_{1-i}(u)+\Lambda_{1-i}(v)-i)
    \\&=\max_{i\le 1+\max(\til{\FS}(u),\til{\FS}(v))}(\Lambda_i(u)+\Lambda_i(v)+i-1),
    \end{align*}
    as desired.
\end{proof}

\subsection{Proof of Theorem~\ref{thm:minimal-flag-variety}}

Recall that for $w\in S_{\Z_+}$, $\FS(w) = \til{\FS}(w)$, so the right sides of \eqref{eq:forward-stabilization} and \eqref{eq:mfv-formula} are the same. Therefore, Theorem~\ref{thm:minimal-flag-variety} is a direct consequence of Corollary~\ref{cor:forward-stabilization} and the following lemma:

\begin{lemma} \label{lem:FS-is-stable}
For all $u,v\in S_{\Z_+}$, $\FS(u,v) = \til{\FS}(u,v)$.
\end{lemma}

\begin{proof}
    One inequality holds by definition:
    \[
    \FS(u,v) = \max_{w|c_{u,v}^w\ne 0} \til{\FS}(w) \le \max_{w|\ol{c_{u,v}^w}\ne 0} \til{\FS}(w) = \til{\FS}(u,v).
    \]
    
    For the other direction, we will consider the Schubert product $\Sch_{\iota w}\Sch_{\iota v}$, and show that there exists some $w\in S_\Z$ with $\ol{c_{\iota u,\iota v}^w}\ne 0$, such that $\til{\FS}(w) \le 0$ and $\til{\BS}(w) = \til{\BS}(\iota u,\iota v)$. Then the permutation $\iota w$ is an element of $S_{\Z_+}$ and satisfies $c_{u,v}^{\iota w} \ne 0$, $\til{\FS}(\iota w) = \til{\FS}(u,v)$, and therefore, $\FS(u,v) \ge \til{\FS}(\iota w) = \til{\FS}(u,v)$.
 
    To find such a permutation, we use a similar approach to Lemma \ref{lem:crucial-last-step}. Let $k = \ell(u) + \ell(v) = \ell(\iota u) + \ell(\iota v)$ and $a = \max(\til{\BS}(u),\til{\BS}(v))$. Fix $i\ge -a$ and $\bs{p} \in \Schvec(\iota u)\Schvec(\iota v)$ such that $I_i(\bs{p}) - i$ is maximal; by Proposition~\ref{prop:FS-BS-relationship}(b,c), $I_i(\bs{p}) - i = \til{\BS}(\iota u,\iota v)$. 
    The maximality of $I_i(\bs{p})-i$ ensures that either $\IncSuf_i(\bs{p}) = \emptyset$ or the last character of $\bs{p}$ has value $i$. In the former case, $i=-j$, so $\til{\BS}(\iota u, \iota v) = \max(\til{\BS}(\iota u), \til{\BS}(\iota v))$, so $\til{\FS}(u, v) = \max(\til{\FS}(u), \til{\FS}(v))$. Since (non-back-stable) Schubert expansions are nonempty, and by DC-triviality, $\FS(u,v) \geq \til{\FS}(u),\til{\FS}(v)$, so the lemma holds in this case.
    Thus, we can assume we are in the latter case, and so $i<0$ since $\iota u,\iota v\in S_{-\N}$.
    
    Recall by \eqref{eq:Schub-word-equivalence} that $\Schvec(u)\Schvec(v) \equiv \sum_w \ol{c_{u,v}^w} \Schvec(w)$. By Theorem \ref{thm:nabla-xi-preseve-equiv} and Corollary \ref{cor:iota-cuvw-relation},
    \[
    \nabla^{k - I(\bs{p})}\left(\Schvec(\iota u)\Schvec(\iota v)\right) \equiv \nabla^{k - I(\bs{p})}\left(\sum_w \ol{c_{u,v}^w} \Schvec(\iota w)\right).
    \]

    Applying the definition of $\nabla$,
    \begin{equation} \label{eq:pi-iota-essential-step}
    \sum_{\bs{q} \in \Schvec(\iota u)\Schvec(\iota v)}\mathcal{\pi}(\bs{q}_{1,k-I(\bs{p})})\bs{q}_{k - I(\bs{p}) + 1, k}
    \equiv \sum_{w} \ol{c_{u,v}^w} \sum_{\bs{w} \in \RW(\iota w)} \mathcal{\pi}(\bs{w}_{1,k-I(\bs{p})})\bs{w}_{k - I(\bs{p}) + 1, k},
    \end{equation}
    where for any $\bs{q}$, $\pi(\bs{q}) = \val(\bs{q}_1) \val(\bs{q}_2) \cdot \ldots \cdot \val(\bs{q}_{\ell(\bs{q})})\in\Z$.

    $\bs{p}$ satisfies the hypotheses to apply
    Proposition \ref{prop:FS-BS-relationship}(d), and $I_i(\bs{p}) \neq 0$ so
    \begin{align*}
    m(\bs{p})_{k - (I(\bs{p}) - 1), k} &= \IncSuf(m(\bs{p}))  \\
    &= (i - (I(\bs{p}) - 1), \ldots, i - 1, i) \\
    &= (-\til{\BS}(\iota u, \iota v) + 1, \ldots, i - 1, i).
    \end{align*}
    
    Note that for any $\bs{q} \in \Schvec(\iota u) \Schvec(\iota v)$, $\val(\bs{q}_j) < 0$, so the sign of every term on the left side of \eqref{eq:pi-iota-essential-step} is $(-1)^{k-I(\bs{p})}$, and since this sign is independent of $\bs{q}$, there is no cancellation. $\bs{p}$ is in this summation, so $T_{(-\til{\BS}(\iota u, \iota v) + 1, \ldots, i - 1, i)}$ of the left hand side of \eqref{eq:pi-iota-essential-step} is nonzero, and the same must be true for the right hand side.

    This means that there must exist some $w\in S_\Z$ with $\ol{c_{u,v}^w}\ne 0$ and some $\bs{w} \in \RW(\iota w)$ with $\mathcal{\pi}(\bs{w}_{1,k-I(\bs{p})}) \neq 0$ and $m(\bs{w}_{k-I(\bs{p})+1,k}) = (-\til{\BS}(\iota u, \iota v) + 1, \ldots, i - 1, i)$. By Proposition \ref{prop:Z-star-increasing-suffix}, $\IncSuf(\bs{w}) = \IncSuf(m(\bs{w}))$, so $\bs{w}_{k-I(\bs{p})+1,k} = m(\bs{w}_{k-I(\bs{p})+1,k}) = (-\til{\BS}(\iota u, \iota v) + 1, \ldots, i - 1, i)$, so $0 \notin \supp\bs{w}$ since it doesn't appear in either $\bs{w}_{1,k-I(\bs{p})}$ or $\bs{w}_{k-I(\bs{p})+1,k}$.
    
    Since $0\notin\supp\bs{w}$, we can also conclude that $j\notin\supp\bs{w}$ for any $j>0$. If there were some positive $j\in\supp\bs{w}$, we could use the Coxeter relations to move $j$ to the end fo the word and obtain some $\bs{w}'\in\RW(w)$ ending in $j$. This would mean $w$ has a descent at $j$, which contradicts Lemma~\ref{lem:Knutson-dc-triv}.
    
    Therefore, $\til{\FS}(w)\le 0$, and since $1-\til{\BS}(\iota u, \iota v)\in\supp\bs{w}$, $\til{\BS}(w) = \til{\BS}(\iota u, \iota v)$,
    and $w$ is the desired permutation.
\end{proof}

\section{Down-up connectedness} \label{sec:down-up-connectedness}
Let $t_{a,b}$ be the permutation written in cycle notation as $(a\; b)$. Permutations $w, w' \in S_{\Z}$ are said to be connected by a \textit{down-up Bruhat move}, denoted $w \text{---} w'$, if for some $a,b,c,d \in \Z$, $wt_{a,b}t_{c,d} = w'$ with $\ell(w) = \ell(wt_{a,b}) + 1 = \ell(wt_{a,b}t_{c,d})$. It can be seen that this relation is symmetric but not transitive. A set of permutations $S$ is said to be down-up Bruhat connected if for all $s,t \in S$, there exist some $a_1, \ldots, a_n \in S$ such that $s \text{---} a_1 \text{---} \ldots \text{---} a_n \text{---} t$.

\begin{conjecture} \label{conj:downup}
    For any $u,v \in S_{\Z}$, $\{w \;|\; \ol{c_{u,v}^w} \neq 0\}$ is down-up Bruhat connected.
\end{conjecture}

\begin{conjecture} \label{conj:nonback-stableDownup}
    For any $u,v \in S_{\Z_+}$, $\{w \;|\; c_{u,v}^w \neq 0\}$ is down-up Bruhat connected.
\end{conjecture}

Permutations $w, w' \in S_{\Z}$ are said to be connected by a \textit{down-up Monk move}, denoted $w \overset{M}{\text{---}} w'$, if there is a down-up Bruhat move $wt_{a,b}t_{c,d} = w'$ such that either $a \leq c < b \leq d$ or $c \leq a < d \leq b$. Equivalently, $w, w'$ are said to be connected by a down-up Monk move if there exists some $u \in S_{\Z}$ and some $i \in \Z$ such that $\ol{c_{u, s_i}^w} \neq 0$ and $\ol{c_{u, s_i}^{w'}} \neq 0$. The fact that these two conditions are equivalent follows from Monk's rule \cite{monk1959geometry}. A set of permutations $S$ is said to be \textit{down-up Monk connected} if for all $s,t \in S$, there exist some $a_1, \ldots, a_n \in S$ such that $s \overset{M}{\text{---}} a_1 \overset{M}{\text{---}} \ldots \overset{M}{\text{---}} a_n \overset{M}{\text{---}} t$.

\begin{conjecture} \label{conj:down-up MonkConnected}
    For any $u, v \in S_{\Z}$, $\{w \;|\; \ol{c_{u,v}^w} \neq 0\}$ is down-up Monk connected.
\end{conjecture}
\begin{conjecture} \label{conj:nonback-stabledown-up MonkConnected}
    For any $u, v \in S_{\Z_+}$, $\{w \;|\; c_{u,v}^w \neq 0\}$ is down-up Monk connected.
\end{conjecture}
If Conjecture \ref{conj:down-up MonkConnected} is true, then there is a simple proof that $\BS(u,v) = \St(u,v)$.
\begin{lemma}
    Given $w \overset{M}{\text{---}} w'$, $|\BS(w) - \BS(w')| \leq 1$.
\end{lemma}
\begin{proof}
    Given $u$ and $i$ such that $\ol{c_{u, s_i}^w} \neq 0$ and $\ol{c_{u, s_i}^{w'}} \neq 0$, it suffices to show that $\BS(w), \BS(w') \in \{\max(\BS(u), 1-i), 1 + \max(\BS(u), 1-i)\}$.

    $\ol{c_{u,s_i}^w} \neq 0$ implies that $u < w$ in the Bruhat order, so any reduced word of $u$ is a subword of a reduced word for $w$. Similarly, $s_i < w$. So, $\max(\BS(u), 1-i) \leq \BS(w)$.

    Let $x$ be the minimum entry in any reduced word for $u$ (equivalently, the minimum index such that $\theta_x(u) = 1$). Let $y  = \min(x, i)$. For any $z < y - 1$, assume $\theta_z(w) = 1$. If $\theta_{z+1}(w) = 0$, then $w$ has a descent at $z$ but $u$ and $s_i$ don't which contradicts DC-triviality (Lemma \ref{lem:Knutson-dc-triv}). If $\theta_{z+1}(w) \neq 0$, then pick $\bs{w} \in \RW(w)$. $z \in \bs{w}$, $z + 1 \in \bs{w}$, and a reduced word for $u$ is a subword of $\bs{w}$. So $\ell(w) = \ell(\bs{w}) \geq 2 + \ell(u)$ which is a contradiction. So, for all $z < y-1$, $\theta_z(w) = 0$. This implies that $\BS(w) \leq 1 - (y - 1) = 1 + (1 - \min(x, i)) \leq 1 + \max(\BS(u), 1-i)$.
    
    The same reasoning shows that $\max(\BS(u), 1-i) \leq \BS(w') \leq 1 + \max(\BS(u), 1-i)$.
\end{proof}
\begin{proposition}
    Conjecture \ref{conj:down-up MonkConnected} implies $\BS(u,v) = \St(u,v)$ for all $u,v \in S_{\Z_+}$.
\end{proposition}

\begin{proof}
    Let $w$ be a permutation such that $\ol{c_{u,v}^w} \neq 0$ and $\BS(w) = \BS(u,v)$. Let $w'$ be a permutation such that $\BS(w') = 0$ and $\ol{c_{u,v}^{w'}} \neq 0$. By down-up Monk connectedness, we can find $a_1, \ldots, a_n \in \{w \;|\; \ol{c_{u,v}^w} \neq 0\}$ such that $w \overset{M}{\text{---}} a_1 \overset{M}{\text{---}} \ldots \overset{M}{\text{---}} a_n \overset{M}{\text{---}} w'$. Then we have a list of permutations where the distance in back-stabilization number between adjacent permutations is at most $1$, so in this list we must have a permutation with back-stabilization number $i$ for all $\BS(u,v) = \BS(w) \leq i \leq \BS(w') = 0$, so $\BS(u,v) = \St(u,v)$.
\end{proof}

\begin{remark}
    Conjectures \ref{conj:downup}, \ref{conj:nonback-stableDownup}, \ref{conj:down-up MonkConnected}, and \ref{conj:nonback-stabledown-up MonkConnected} are partially ordered in terms of strength with $\ref{conj:nonback-stabledown-up MonkConnected} \implies \ref{conj:down-up MonkConnected} \implies \ref{conj:downup}$ and $\ref{conj:nonback-stabledown-up MonkConnected} \implies \ref{conj:nonback-stableDownup} \implies  \ref{conj:downup}$. Two of these implications are because down-up Monk connectedness is a stricter condition that down-up connectedness. The other two are because applying $\gamma$ preserves down-up (Monk) connectedness, so if the non-back-stable structure constants are down-up (Monk) connected, we can apply appropriate shifting to show that the back-stable structure constants are down-up (Monk) connected.
\end{remark}

\section{Back-stable key polynomials} \label{sec:back-stable-keys}

Key polynomials, also known as Demazure characters, are indexed by \emph{positive compositions}, finite sequences $\alpha = (\alpha_1,\alpha_2,\ldots)$ of nonnegative integers, and are denoted $\kappa_{\alpha}$. $\alpha$ is said to have $\alpha_i$ \emph{parts} of size $i$, and the word ``positive'' in our terminology refers to the fact that all parts of $\alpha$ are positive.

Key polynomials were originally studied by Demazure \cite{demazure1974nouvelle} in geometric contexts, and later more combinatorially by Lascoux and Sch\"utzenberger.

Let $\alpha$ be a positive composition, and let $\lambda$ be the partition formed by reordering of the parts of $\alpha$. Let $u(\alpha)$ be the shortest-length permutation such that $\alpha = u(\alpha)\lambda$, where the action is by permutation of parts.

Then the Demazure character is given by:
\[
\kappa_\alpha = \pi_{u(\alpha)} x^\lambda, \qquad\qquad x^\lambda = x_1^{\lambda_1}x_2^{\lambda_2}\cdots,
\]
where $\pi_{s_{i_1}\cdots s_{i_n}} = \pi_{i_1}\cdots \pi_{i_n}$, and $\pi_i$ is the isobaric divided-difference operator
\[
\pi_i = \frac{1-s_i}{x_i-x_{i+1}} x_i.
\]

This definition does not lend itself to back-stabilization: the unit shift $\gamma$ adds a leading 0 to $\alpha$, so $u(\alpha) < u(\gamma(\alpha))$, and iterating $\gamma$ causes $u$ to grow without bound.

Instead, we use an alternate formula for key polynomials due to Reiner and Shimozono \cite{ReinerShimozono}. The \textit{nilplactic equivalence} relation on words, denoted with $\nilsim$, is the symmetric transitive closure of

\begin{equation}
\begin{aligned}
    \bs{a} \circ (i, i+1, i) \circ \bs{b} &\; \nilsim \; \bs{a} \circ (i+1, i, i+1) \circ \bs{b}, \\
    \bs{a} \circ (x,z,y) \circ \bs{b} &\; \nilsim \; \bs{a} \circ (x,y,z) \circ \bs{b}, \\
    \bs{a} \circ (y,x,z) \circ \bs{b} &\; \nilsim \; \bs{a} \circ (y,z,x) \circ \bs{b}
\end{aligned}
\end{equation}

A \textit{column strict tableau} is a filling of a Ferrers diagram with integers which is weakly increasing in each row and strictly increasing down each column. A \textit{positive column strict tableau} is defined similarly, with the extra the constraint that the integers must be positive. The \textit{column reading word} of the column strict tableau $T$ is $\text{column}(T) = \bs{v_1} \circ \bs{v_2} \circ \ldots \circ \bs{v_j}$ where $\bs{v_j}$ is the column word (strictly decreasing word) comprising the $j$th column of $T$. The \textit{row reading word} of $T$, $\text{row}(T) = \bs{v_1} \circ \bs{v_2} \circ \ldots \circ \bs{v_j}$, is the row word (strictly increasing word) comprising the $j$th row of $T$. The \textit{content} of $T$ is the composition $\text{content}(T)$ such that $\text{content}(T)_i$
is the number of occurrences of the letter $i$ in $T$. Note that $\text{content}(T)$ is a positive composition if and only if the entries of $T$ are positive.

Define \textit{row insertion} of the word \( \bs{v} = (v_1, v_2, \cdots, v_n) \) into the column strict tableau \( T \) as follows. Let \( P_k \) denote the unique column strict tableau such that 
\[
\operatorname{row}(P_k) \nilsim \operatorname{row}(T) \circ \bs{v}_{1,k},
\]
for \( 0 \leq k \leq n \) (where \( P_0 = T \)). Then write \( (T \leftarrow \bs{v}) = P_n \).

Let \( v \) be a word. Consider the \textit{column word factorization} of \( \bs{v} \), i.e., writing
\[ \bs{v} = \bs{v_1} \circ \bs{v_2} \circ \cdots, \]
where each \( \bs{v_i} \) is a maximal column word. Define the \textit{column form} of \( \bs{v} \) to be the positive composition with $\text{colform}(\bs{v})_i = \ell(\bs{v_i})$. Let \( (\varnothing \leftarrow \bs{v}) = P \). Say that \( \bs{v} \) is a \textit{column-frank word} if \( \text{colform}(\bs{v}) \) is a rearrangement of the nonzero parts of \( \lambda' \), where \( \lambda \) is the shape of \( P \) and \( \lambda' \) denotes the conjugate shape of \( \lambda \). 

Let \( P \) be a column strict tableau of shape \( \lambda \). The \textit{left nil key} of \( P \), denoted \( K^{\bullet}_-(P) \) is the key of shape \( \lambda \) whose \( j \)th column is given by the first column of any column-frank word \( \bs{v} \) such that \( \bs{v} \nilsim P \) and \text{colform}( \( \bs{v} \) ) is of the form \( (\lambda'_j, \ldots) \). One can show that \( K^{\bullet}_-(P) \) is well defined.

Reiner and Shimozono showed that \cite[Theorem~5]{ReinerShimozono}
\begin{equation}
    \kappa_\alpha = \sum_{\text{rev}(\bs{a}) \nilsim P} \F_{\bs{a}}
\end{equation}
Where $P$ is any positive column strict tableaux with a reduced column reading word and which satisfies $\text{content}(K^{\bullet}_-(P)) = \alpha$.

The right side of this formula back-stabilizes by replacing $\F_{\bs{a}}$ with $\olF_{\bs{a}}$, leading naturally to the following definition for back-stable key polynomials. (We thank Vic Reiner for suggesting this definition).

We now relax our definitions of composition and column strict tableaux, so that they may have nonpositive parts. Otherwise, all definitions above are unchanged. We define.
\begin{definition}
    The \textit{back-stable key polynomial} for composition $\alpha$ is
    \begin{equation}
    \olK_\alpha = \sum_{\text{rev}(\bs{a}) \nilsim P} \olF_{\bs{a}},
\end{equation}
Where $P$ is any column strict tableaux with a reduced column reading word and which satisfies $\text{content}(K^{\bullet}_-(P)) = \alpha$.
\end{definition}

If $\alpha$ has nonpositive parts, repeated shifts $\gamma^k$ results in a positive composition. For large enough $k$, we have:
\begin{equation} \label{eq:key-stabilization}
\gamma^k(\olK_\alpha)|_{\cdots = x_{-2} = x_{-1} = x_0 = 0} = \kappa_{\gamma^k(\alpha)}.
\end{equation}

\begin{proposition}
    $\olK_\alpha$ is back-symmetric.
\end{proposition}

\begin{proof}
This follows from the fact that the stable limits of key polynomials are Schur functions (see e.g. \cite[Corollary~4.9]{AssafSearles-Kohnert}).
\end{proof}

\begin{proposition}
Back-stable key polynomials form a basis for $\ol{R}$.
\end{proposition}

\begin{proof}
Linear independence follows from the linear independence of the $\kappa_\alpha$ since if $\alpha$ is a positive composition, $\olK_\alpha\mapsto\kappa_\alpha$ under the specialization $\cdots = x_{-2}=x_{-1}=x_0=0$. If a composition in the linear combination is not positive, apply \eqref{eq:key-stabilization}.

The $\olK_\alpha$ span $\ol{R}$ since the back-stable Schubert polynomials $\ol{\Sch}_w$ span $\ol{R}$ and
\[
\ol{\Sch}_w = \sum_\alpha \olK_\alpha,
\]
where the sum is over all compositions $\alpha$ such that there exists $\bs{a}\in\RW(w)$ with $\text{rev}(\bs{a})\nilsim\text{key}(\alpha)$.
\end{proof}

\begin{remark}
The expansion formula for back-stable Schuberts into back-stable keys is the same as in the non-back-stable case, since both families have the same expansions in terms of slide polynomials.
\end{remark}

Finally, we conjecture a formula for the action of $\xi$ on $\olK_\alpha$. Given some composition $\alpha$, let $\eta(\alpha) = \{i \;|\;\; \nexists j < i, \; \alpha_j = \alpha_i\}$. Let $\delta_i$ be the composition be composition which is $1$ at index $i$ and $0$ everywhere else. Define composition addition and subtraction component-wise.

\begin{conjecture}
    For any composition $\alpha$, $$\xi(\olK_\alpha) = \sum_{i \in \eta(\alpha)} \olK_{\alpha - \delta_i}$$
\end{conjecture}

This conjecture holds for Schur polynomials (i.e. the case where $\alpha$ is a reverse partition). We have checked that it holds for all compositions of size at most 6.

\bibliographystyle{siam}
\bibliography{bibliography.bib}

\end{document}